\documentclass[aip,reprint]{revtex4-2}

\usepackage{amsmath,amssymb,amsthm}
\usepackage{graphicx} 
\usepackage[dvipsnames]{xcolor}

\newtheorem{lemma}{Lemma}
\newtheorem{definition}{Definition}
\newtheorem{theorem}{Theorem}
\newtheorem{proposition}{Proposition}

\newcommand{\R}{\mathbb{R}}

\newcommand{\Z}{\mathbb{Z}}

\newcommand{\khat}{\hat{k}}
\newcommand{\xhat}{\hat{x}}
\renewcommand{\vec}[1]{\mathbf{#1}}

\newcommand{\iin}{\mathrm{in}}
\newcommand{\out}{\mathrm{out}}

\newcommand{\nnodes}{N} 
\newcommand{\jnodes}{J} 
\newcommand{\lnodes}{L}

\begin{document}


\title{Stability of heteroclinic cycles in rings of coupled oscillators} 



\author{Claire M.~Postlethwaite}
\affiliation{Department of Mathematics, University of Auckland}
\author{Rob Sturman}
\affiliation{School of Mathematics, University of Leeds}


\date{\today}

\begin{abstract}
Networks of interacting nodes connected by edges arise in almost every branch of scientific enquiry. The connectivity structure of the network can force the existence of invariant subspaces, which would not arise in generic dynamical systems. These invariant subspaces can result in the appearance of robust heteroclinic cycles, which would otherwise be structurally unstable. Typically, the dynamics near a stable heteroclinic cycle is non-ergodic: mean residence times near the fixed points in the cycle are undefined, and there is a persistent slowing down. In this paper, we examine ring networks with nearest-neighbour or nearest-$m$-neighbour coupling, and show that there exist classes of heteroclinic cycles in the phase space of the dynamics. We show that there is always at least one heteroclinic cycle which can be asymptotically stable, and thus the attracting dynamics of the network are expected to be non-ergodic. 
We conjecture that much of this behaviour persists in less structured networks and as such, non-ergodic behaviour is somehow typical.
\end{abstract}


\maketitle 

\section{Introduction}


{\bf From the coupled map lattices proposed in the 1980s to the modern discipline of complex networks the study of simple systems connected in some way forms a fundamental paradigm in dynamical systems. Applications are plentiful and diverse, and include spatially extended systems, chemical reactions, biological and ecological networks, and span many length scales\cite{strogatz2001exploring}. In many cases, a goal of study is {\em spatiotemporal chaos}\cite{kaneko1989spatiotemporal}, and this might typically mean computing a long-term average (for example, a Lyapunov exponent). However, it is well-known that a class of networks --- those with invariant subspaces forced by symmetries in the system --- permit {\em heteroclinic cycles\cite{guckenheimer1988structurally}}, that is, trajectories along which time averages do not converge, instead slowing down as they repeatedly and systematically get closer and closer to fixed points. We investigate a family of coupled map lattices defined on ring networks and establish stability properties of the many possible families of heteroclinic cycles. }

Specifically, we consider how the structure of the architecture, or topology, of the network of physical nodes determines the architecture of a heteroclinic network in phase space between fixed points. We note that this question was asked in the reverse by Ashwin and Postlethwaite~\cite{ashwin2013designing}, who showed how to construct a system of ordinary differential equations into which was embedded a heteroclinic network of any specified topology.

In particular, we consider systems of the form
\begin{equation}\label{eq:gen_coupled} 
x_{i+1}^{(k)} = f(x_i^{(k)}) e^{-\gamma \sum_{\khat} A_{\khat k} x_i^{(\khat)}},\quad k=1,\dots, \nnodes,
\end{equation}
where $\gamma\le 0$ is a parameter. Each of the $x_i^{(k)}\in[0,1]$ measures the activity at time $i$ of the $k$th `node' in a network, and $A_{\khat k}$ is an adjacency matrix which determines the connectivity between the different nodes. We take the function $f$ to be the logistic map
\[
f(x)=rx(1-x)
\]
where $r\in(0,4]$ is a parameter. The dynamics of the uncoupled system with $\gamma =0$
is well known~\cite{brin2002introduction}, but briefly, for $r\in(0,1]$, the origin is an asymptotically stable fixed point. At $r=1$ there is a transcritical bifurcation creating a second fixed point at $x=\frac{r-1}{r}$. This fixed point is asymptotically stable for $r\in (1,3]$, and at $r=3$ undergoes a period-doubling bifurcation which leads to a period-doubling cascade followed by the onset of chaos at $r\approx 3.56995$. 

When $x^{(\hat{k})}$ is non-zero for some $\hat{k}$, then the term $e^{-\gamma \sum_{\khat} A_{\khat k} x^{(\khat)}}$ in~\eqref{eq:gen_coupled}, with $\gamma >0$,  has an inhibitory effect upon any node connected to $\hat{k}$, i.e.~those for which $A_{\khat k}=1$. Specifically, if $\gamma x^{(\hat{k})}$ is large enough, then this term can have the same effect as reducing the value of $r$ in the $x^{(k)}$ equation to less than 1, and hence causing the values of those $x^{(k)}$ to decrease towards zero. Heuristically, it is then clear that oscillatory behaviour is possible, as nodes can alternately be active (have a non-zero value), and hence inhibit those nodes they are connected to; decay, when other nodes in turn inhibit them; and finally grow again to an active state as the nodes inhibiting them decay in turn.

In figure~\ref{fig:three} we show a time series from a cycle of three such coupled nodes, specifically, the set of equations
\begin{align}
x_{i+1}^{(1)} &= f(x_i^{(1)} ) e^{-\gamma x_i^{(3)}}, \nonumber \\
x_{i+1}^{(2)}  &= f(x_i^{(2)}) e^{-\gamma x_i^{(1)}}, \label{eq:GH} \\
x_{i+1}^{(3)} &= f(x_i^{(3)}) e^{-\gamma x_i^{(2)}}. \nonumber 
\end{align} 
In panel (a), we use $r=2$, so the dynamics of the uncoupled system contains a stable fixed point. The time series clearly shows the trajectory cycling between three fixed points, in a manner essentially identical to that seen in the well-known Guckenheimer--Holmes heteroclinic cycle\cite{guckenheimer1988structurally}. In panel (b), we use $r=3.5$, so the uncoupled system is in the chaotic regime, and we see cycling between three chaotic attractors. This phenomena was previously described by Ashwin, Rucklidge and Sturman~\cite{ashwin2002infinities,ashwin2003decelerating}.

\begin{figure}[tb]
\setlength{\unitlength}{1mm}
\begin{center}
\begin{picture}(85,70)(0,0)
\put(0,0){\includegraphics[trim=5mm 0mm 5mm 5mm,clip=true,width=85mm]{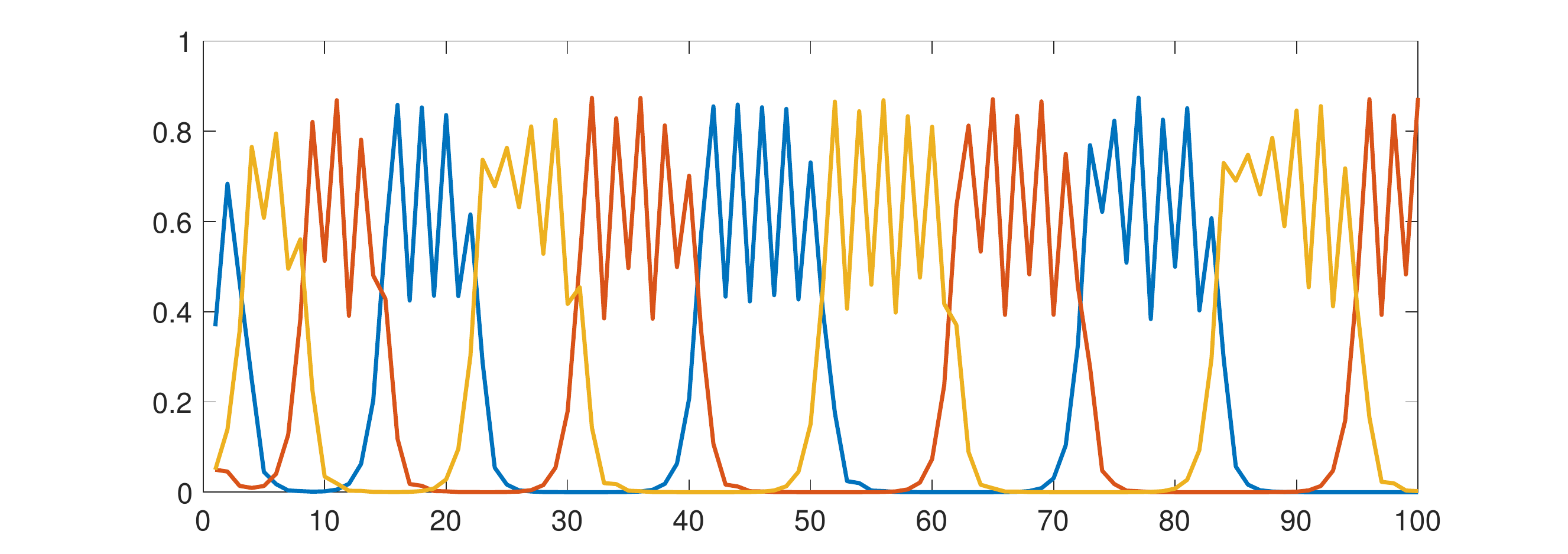}}
\put(0,35){\includegraphics[trim=5mm 0mm 5mm 5mm,clip=true,width=85mm]{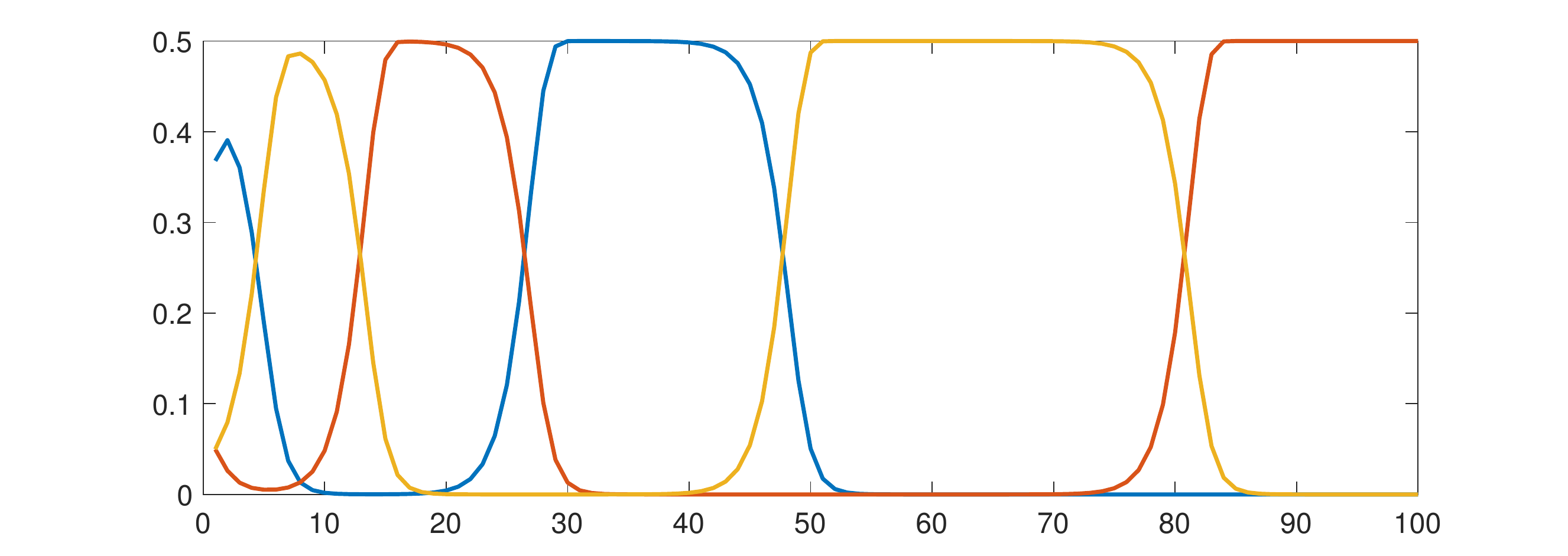}}

\put(0,20){\rotatebox{90}{$x^{(k)}$}}
\put(0,55){\rotatebox{90}{$x^{(k)}$}}

\put(73,0){$i$}
\put(73,35){$i$}

\put(0,65){(a)}
\put(0,30){(b)}

\end{picture}
\end{center}
\caption{\label{fig:three} The figures show time series for equations~\ref{eq:GH}, where components $x^{(1)}$, $x^{(2)}$ and $x^{(3)}$ are shown by the blue, red and yellow lines respectively. In panel (a), $r=2$, $\gamma=3.5$; in panel (b), $r=3.5$, $\gamma=3.5$.}
\end{figure}
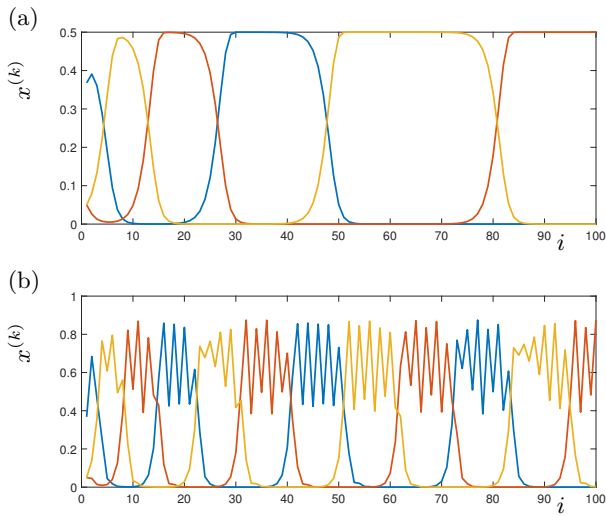

In this paper, we extend the work of Ashwin et al.~\cite{ashwin2002infinities,ashwin2003decelerating} and consider larger networks of coupled systems in the form of~(\ref{eq:gen_coupled}). We refer these equations as describing the network of connections between nodes in \emph{physical space}, and for the remainder of the paper, refer to this network as instead a \emph{directed graph} with \emph{directed edges} between \emph{nodes}.
We begin in section~\ref{sec:egfive}, by considering another example: a five-node ring graph with one-way nearest-neighbour coupling. We determine the fixed points and the heteroclinic connections which exist between them. We refer to this network of connection as the \emph{phase space network}, or \emph{heteroclinic network}, which has \emph{heteroclinic connections} (or sometimess simply \emph{connections}) between \emph{fixed points}.
In section~\ref{sec:general}, we consider general systems of the form of~(\ref{eq:gen_coupled}) and describe how to find the fixed points and heteroclinic connections for such a system. In general, this procedure results in a very complex heteroclinic network that is difficult to analyse, so in section~\ref{sec:n1} we look in detail at $\nnodes$-node directed graphs with one-way nearest neighbour coupling in the physical space. Here, as well as determining the structure of the heteroclinic network in phase space, we are able to analyse the dynamic stability of subcycles within the network. We use results from Podvigina~\cite{podvigina2012stability}, and some classical results on solutions to polynomials~\cite{cauchy1829leccons,rouche1866memoire} to prove theorem~\ref{thm:stab_thm}, which shows that only one of the subcycles can ever be stable, and then, only if $\gamma$ is large enough. In section~\ref{sec:nm} we make some conjectures about larger networks. Section~\ref{sec:disc} concludes.

\section{Example: five node ring graph with nearest-neighbour coupling}
\label{sec:egfive}

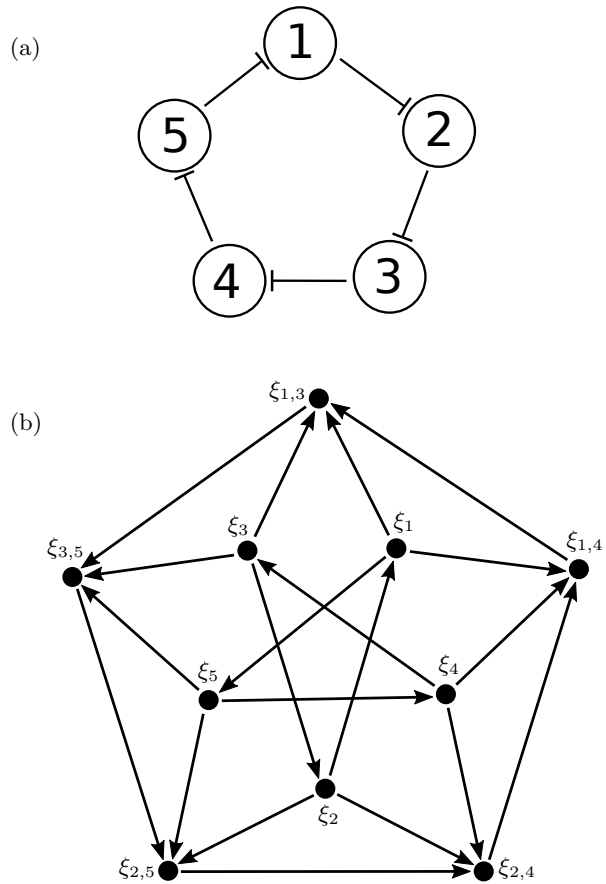
\begin{figure}[tb]
\setlength{\unitlength}{1mm}
\begin{center}
\begin{picture}(85,115)(0,0)
\put(7,0){\includegraphics[width=70mm]{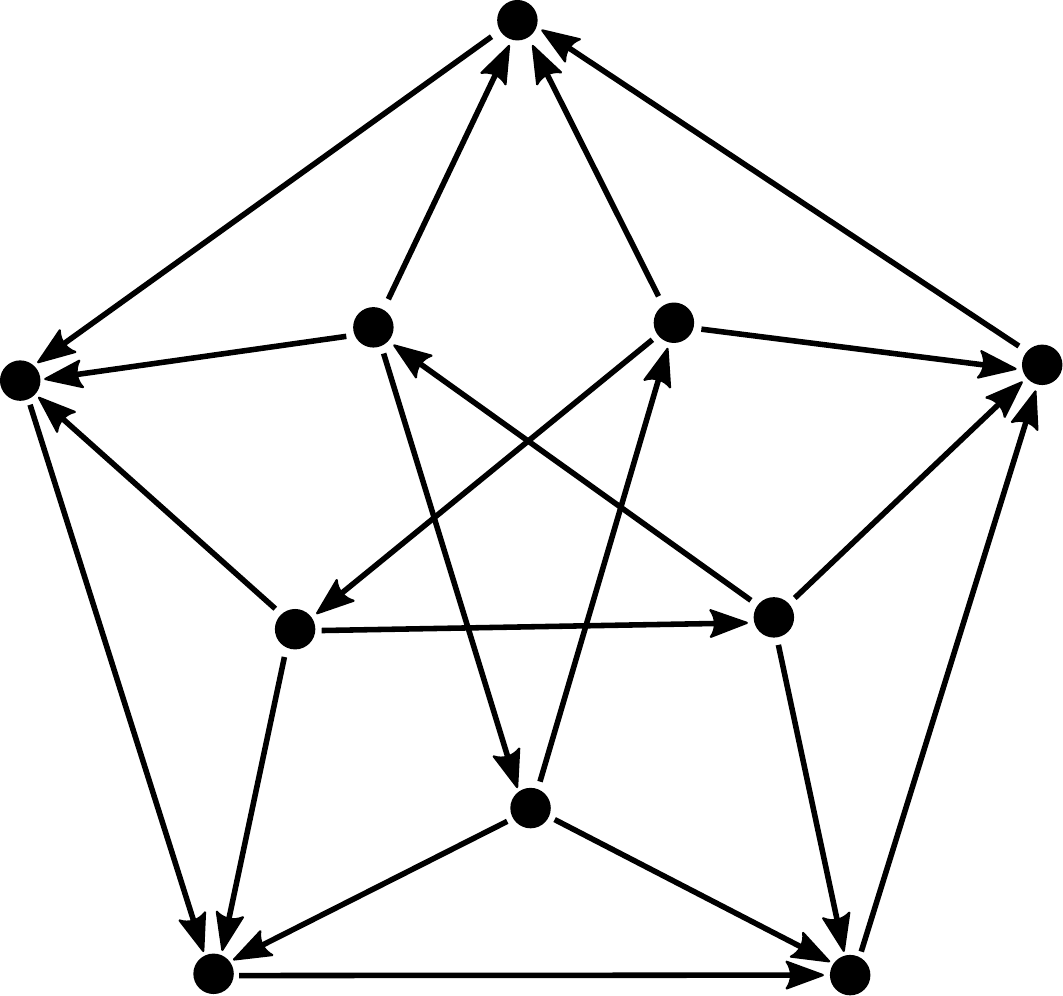}}
\put(17,75){\includegraphics[width=45mm]{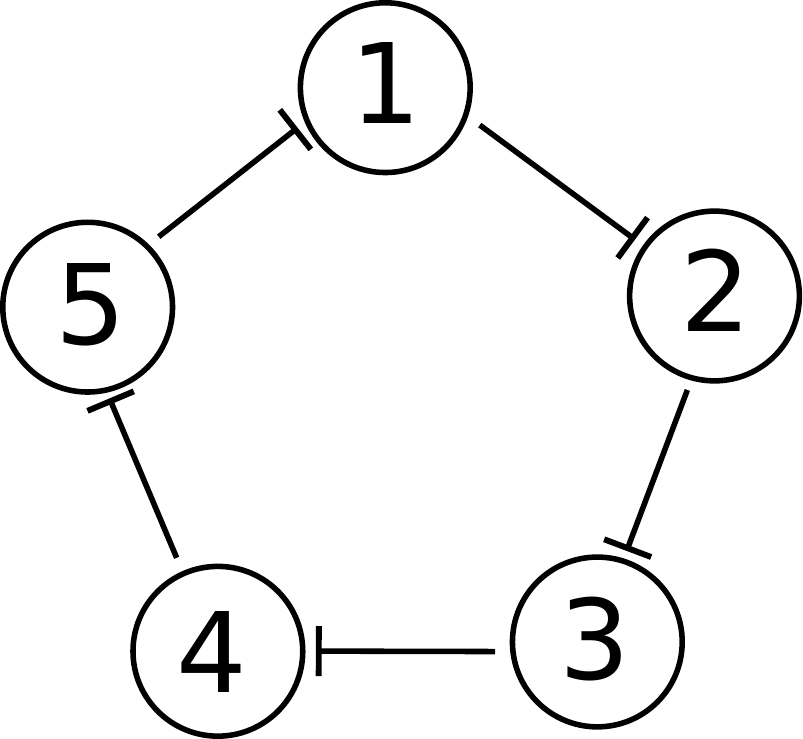}}

\put(51,47){$\xi_1$}
\put(29,46.5){$\xi_3$}
\put(57,28){$\xi_4$}
\put(25,27){$\xi_5$}
\put(41,8){$\xi_2$}

\put(74,44.5){$\xi_{1,4}$}
\put(5,43.5){$\xi_{3,5}$}

\put(65,1){$\xi_{2,4}$}
\put(14.5,1){$\xi_{2,5}$}

\put(34.5,64.5){$\xi_{1,3}$}

\put(0,110){(a)}
\put(0,60){(b)}

\end{picture}
\end{center}
\caption{\label{fig:five} Panel (a) shows the physical network described by equations~\eqref{eq:nnc}, with $\nnodes=5$. Here, the physical nodes are shown by circles, and the inhibitory couplings by flat-ended arrows. Panel (b) shows the corresponding heteroclinic network. The dots represent fixed point solutions of the system~\eqref{eq:nnc}, and arrows indicate the presence of a heteroclinic connection. In this figure and those that follow, note the distinction between circles for the nodes of the physical network, and filled dots for the fixed points of the heteroclinic network in phase space. }
\end{figure}

We give, in this section, an example of a five node directed graph with one-way nearest-neighbour coupling. We determine the possible fixed points in phase space, and the heteroclinic connections between them. We show time series of typical trajectories close to the resulting heteroclinic network but defer the computation of the stability of each of the sub-cycles to section~\ref{sec:n1}.

We begin with a few formal definitions.
Consider the map
\begin{equation}\label{eq:genf}
x_{i+1}=g(x_i),\quad x_i\in\R^n, \quad i\in\Z
\end{equation}
with fixed points $\zeta_1,\dots,\zeta_M$. A \emph{heteroclinic connection} between $\zeta_j$ and $\zeta_{j'}$ is a solution to~\eqref{eq:genf} for which $x_i\rightarrow\zeta_j$ as $i\rightarrow -\infty$ and $x_i\rightarrow\zeta_{j'}$ as $i\rightarrow \infty$. Suppose that there exist heteroclinic connections between $\zeta_{j_k}$ and $\zeta_{j_{k+1}}$ for $k=1,\dots,M-1$, and also one between $\zeta_{j_M}$ and $\zeta_{j_1}$. Then the set $H$ consisting of the fixed points $\zeta_j$ and the connecting orbits is a \emph{heteroclinic cycle}. A \emph{heteroclinic network}~\cite{kirk1994competition} is usually defined to be a union of heteroclinic cycles. In this paper, we relax the definition somewhat: we allow a heteroclinic network to consist of a set of fixed point solutions, and heteroclinic connections between them, which contains at least one heteroclinic cycle. Note that this means that not every heteroclinic connection in the network need be part of a cycle.

For the example we consider in this section, the network of nodes in physical space is shown in figure~\ref{fig:five}(a). The equations governing this system are
\begin{equation}\label{eq:five}
x_{i+1}^{(k)} = f(x_i^{(k)}) e^{-\gamma x_i^{(k-1)}},\quad k=1,\dots, 5,
\end{equation}
which are equivariant with respect to a rotation symmetry of the coordinates.

Equation~\eqref{eq:five} has two different types of fixed points solution in phase space which are of interest to us, namely, those with one node active (that is, with a single component that is $O(1)$), or those with two nodes active. More precisely, assume for now that $r\in [1,3]$, and let $\hat{x}=\frac{r-1}{r}$. Then, using coordinates $(x^{(1)}, x^{(2)}, x^{(3)}, x^{(4)}, x^{(5)})$, 
we label the fixed points with only $x^{(1)}$ active as:
\[
\xi_1=(\hat{x}, 0, 0, 0, 0)
\]
and similarly we have $\xi_2,\dots,\xi_5$, where for each $\xi_j$, the $j$th component is equal to $\hat{x}$, and the remainder are equal to zero. Next, we label 
\[
\xi_{1,3}=(\hat{x}, 0, \hat{x}, 0, 0)
\]
and similarly define $\xi_{j,m}$, which has the $j$th and $m$th components equal to $\hat{x}$, where $|j-m|\neq 1$ (i.e.~$j$ and $m$ are not adjacent nodes in the ring graph). Note that $\xi_{j,m}\equiv \xi_{m,j}$ but we typically list $j$ and $m$ in increasing numerical order.

Note that in this example, there cannot be any fixed points with more than two components equal to $\hat{x}$, because of the connectivity of the graph: two nodes which are connected by an edge cannot both be active at a fixed point. In larger, more general graphs, we would expect to see fixed points with more active components (see section~\ref{sec:general} for the general setup).

We next consider the dynamics of~\eqref{eq:five} in two two-dimensional subspaces, and show that there exist heteroclinic connections from $\xi_j$ to $\xi_{j-1}$  and from $\xi_j$ to $\xi_{j,j \pm 2}$ (where indices are taken $\mod 5$).

For the first, consider the dynamics in the subspace where $x^{(2)}=x^{(3)}=x^{(4)}=0$, namely the system
\begin{equation}\label{eq:five_2da}
\begin{split}
x_{i+1}^{(1)} &=r x^{(1)}_i(1-x^{(1)}_i) e^{-\gamma x_i^{(5)}}, \\
x_{i+1}^{(5)} &= r x^{(5)}_i(1-x^{(5)}_i)
\end{split}
\end{equation}
System~\eqref{eq:five_2da} has fixed points at $(x^{(1)},x^{(5)})=(\hat{x},0)\equiv \xi_1$ and $(x^{(1)},x^{(5)})=(0,\hat{x})\equiv \xi_5$. Consider an initial condition close to $\xi_1$, but with $x^{(5)}_0\neq 0$. Since the $x^{(5)}$ equation is decoupled from $x^{(1)}$, then it behaves as it would in the uncoupled logistic map, specifically, the $x^{(5)}$ component initially grows and approaches the value $\hat{x}$. As $x^{(5)}$ grows, the coupling term in the $x^{(1)}$ equation has the effect of essentially reducing the $r$ value of the logistic map in the $x^{(1)}$ equation to, eventually, $r e^{-\gamma \hat{x}}$. Thus, if $r e^{-\gamma \hat{x}}<1$, then $x^{(1)}$ will eventually decay to zero, and the trajectory approaches $\xi_5$. There is thus a heteroclinic connection from $\xi_1$ to $\xi_5$, and by symmetry, heteroclinic cycles from $\xi_j$ to $\xi_{j-1}$.

For the second type of connection, consider the dynamics in the subspace where $x^{(2)}=x^{(4)}=x^{(5)}=0$, namely the system
\begin{equation}\label{eq:five_2db}
\begin{split}
x_{i+1}^{(1)} &=r x^{(1)}_i(1-x^{(1)}_i), \\
x_{i+1}^{(3)} &= r x^{(3)}_i(1-x^{(3)}_i)
\end{split}
\end{equation}
In this subspace, both $x^{(1)}$ and $x^{(3)}$ are decoupled from each other. There are three fixed points, $(x^{(1)},x^{(3)})=(\hat{x},0)\equiv \xi_1$, $(x^{(1)},x^{(3)})=(0,\hat{x})\equiv \xi_3$, and $(x^{(1)},x^{(3)})=(\hat{x},\hat{x})\equiv \xi_{1,3}$. Both $\xi_1$ and $\xi_3$ are saddle points, and perturbations close to these fixed points will result in trajectories which approach $\xi_{1,3}$. There are thus heteroclinic connections between $\xi_1$ and $\xi_{1,3}$, $\xi_3$ and $\xi_{1,3}$, and by analogy, heteroclinic connections between any $\xi_j$ and $\xi_{j,m}$ or $\xi_{m,j}$.    

Finally, we consider the dynamics of~\eqref{eq:five} in a three-dimensional subspace, and show that there is a heteroclinic connection from $\xi_{j,j+2}$ to $\xi_{j-1,j+2}$. Consider the dynamics in the subspace where $x^{(2)}=x^{(4)}=0$, namely the system
\begin{equation}\label{eq:five_3d}
\begin{split}
x_{i+1}^{(1)}  &=r x^{(1)}_i(1-x^{(1)}_i) e^{-\gamma x_i^{(5)}}, \\
x_{i+1}^{(3)}  &= r x^{(3)}_i(1-x^{(3)}_i), \\
x_{i+1}^{(5)}  &= r x^{(5)}_i(1-x^{(5)}_i).
\end{split}
\end{equation}
There are, as before, fixed points in this system with one component non-zero, but of interest right now are the two fixed points with $(x^{(1)},x^{(3)},x^{(5)})=(\hat{x},\hat{x},0)\equiv \xi_{1,3}$ and $(x^{(1)},x^{(3)},x^{(5)})=(0,\hat{x},\hat{x})\equiv \xi_{3,5}$. Note first that the $x^{(3)}$ and $x^{(5)}$ components are decoupled. Both have stable fixed points at $x^{(i)}=\hat{x}$. Thus perturbations close to $\xi_{1,3}$ will have an $x^{(3)}$ component which remains close to $\hat{x}$, but, as in the case of the connection between $\xi_1$ and $\xi_5$, the $x^{(5)}$ component will grow, and again, so long as $r e^{-\gamma \hat{x}}<1$, the $x^{(1)}$ component will decay to zero. 

Due to the rotational symmetry of the system~\eqref{eq:five}, we thus have three families of heteroclinic connections, namely:
\begin{equation}\label{eq:five_conns}
\begin{split}
\xi_j &\rightarrow \xi_{j-1} \\
\xi_j & \rightarrow \xi_{j,k}, \quad k=j\pm 2 \\
\xi_{j,j+2} & \rightarrow \xi_{j-1,j+2}
\end{split}
\end{equation}
Note that here, indices are taken $\mod 5$. In later sections, indicies are taken $\mod n$, where $n$ is the number of nodes in the graph.
We refer to the families of heteroclinic connections as $1\rightarrow 1$ connections, $1\rightarrow 2$ connections and $2\rightarrow 2$ connections respectively, where a $p\rightarrow q$ connection is a transition from a fixed point with $p$ nodes active to a fixed point with $q$ nodes active.

The complete set of connections between fixed points is shown in panel (b) of figure~\ref{fig:five}. Notice that there are two heteroclinic cycles, one between fixed points of type $\xi_j$, with $1\rightarrow 1$ connections,  and the other between fixed points of type $\xi_{j,k}$, with $2\rightarrow 2$ connections. The $1\rightarrow 2$ connections are not part of any cycles. In figure~\ref{fig:5node_ts} we show numerical simulations showing trajectories close to each of these cycles. We will compute the stability of cycles of these types in general in section~\ref{sec:n1} that follows. We shall see that the cycle between fixed points of type $\xi_{j,k}$ can have some form of stablility, if parameters are chosen correctly, specifically, if $\gamma>\frac{3 \log r}{2\hat{x}}$, but the cycle between fixed points of type $\xi_{j}$ can never be stable. However, if initial conditions are chosen carefully (in a manner described in that section), we can, as seen in panel (b) of figure~\ref{fig:5node_ts}, observe this cycle for a reasonably long period of time.

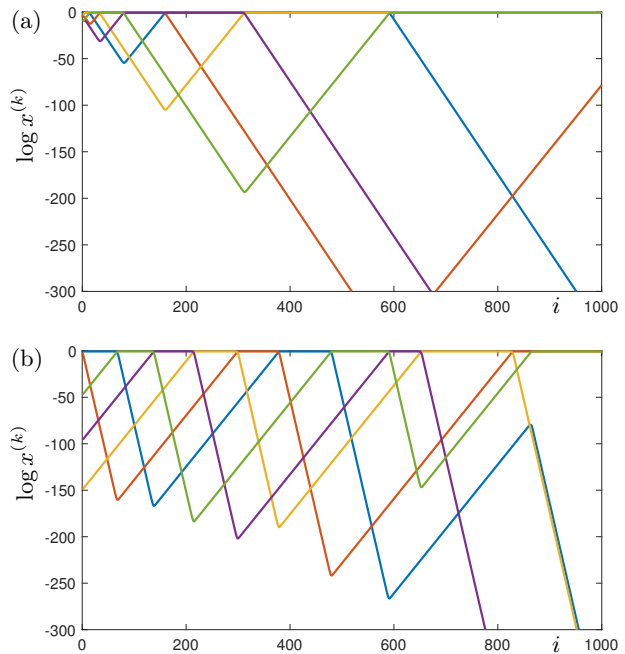
\begin{figure}
\setlength{\unitlength}{1mm}
\begin{center}
\begin{picture}(85,90)(0,0)
\put(0,0){\includegraphics[trim=5mm 5mm 5mm 5mm,clip=true,width=85mm]{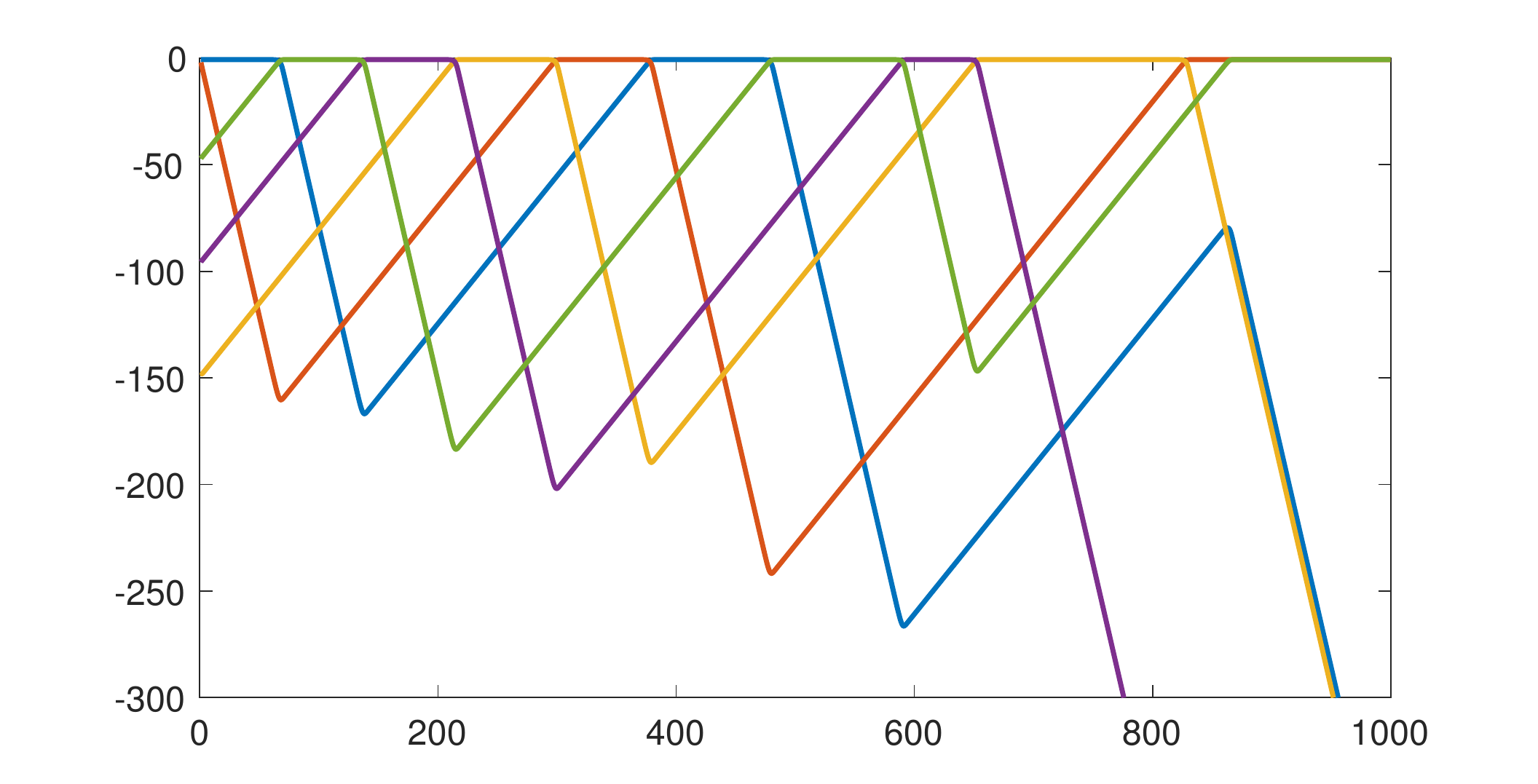}}
\put(0,45){\includegraphics[trim=5mm 5mm 5mm 5mm,clip=true,width=85mm]{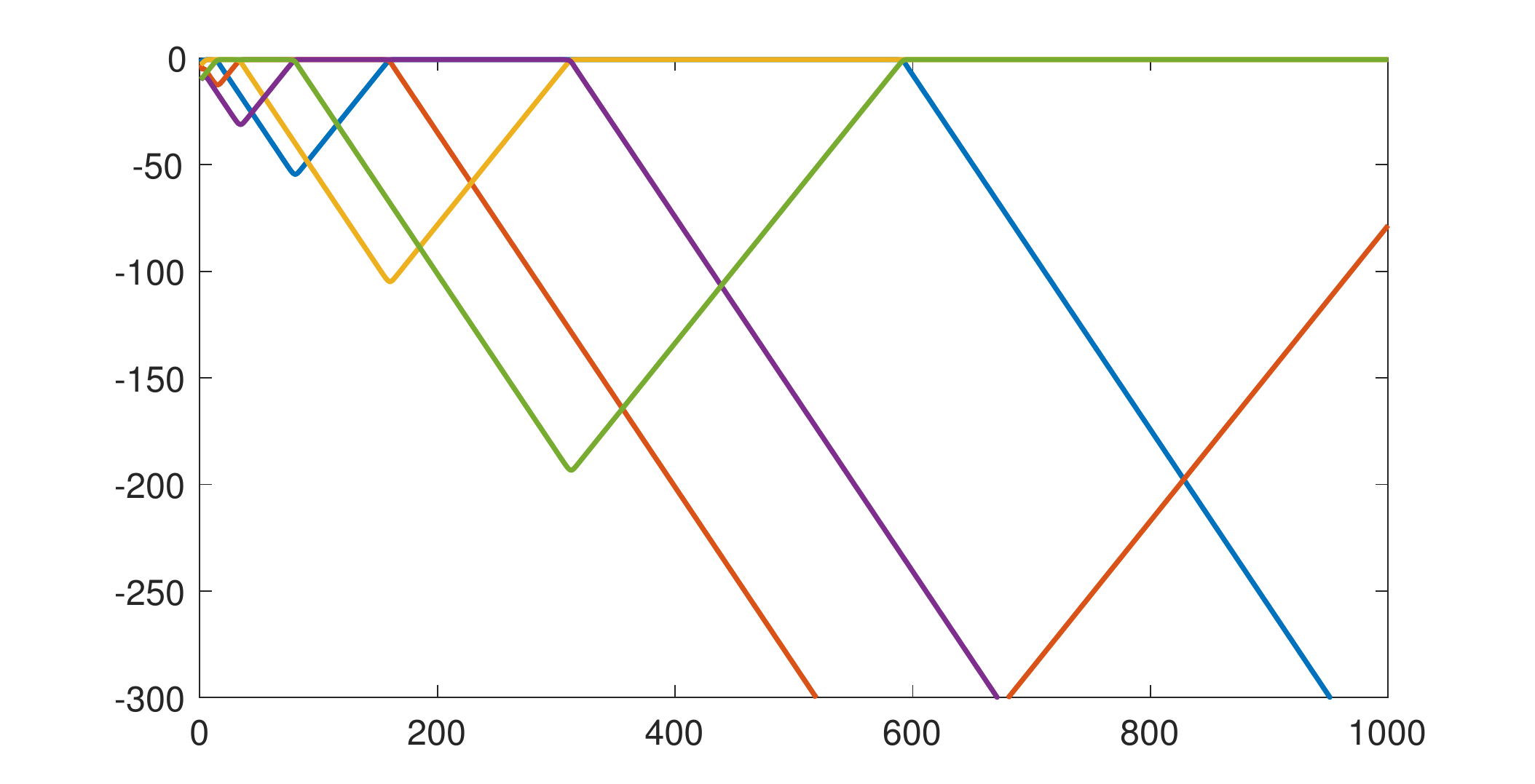}}

\put(0,20){\rotatebox{90}{$\log x^{(k)}$}}
\put(0,65){\rotatebox{90}{$\log x^{(k)}$}}

\put(72,0){$i$}
\put(72,45){$i$}

\put(0,83){(a)}
\put(0,38){(b)}

\end{picture}
\end{center}
\caption{\label{fig:5node_ts} The figure shows trajectories of equation~\eqref{eq:five} cycling between fixed points with (a) two nodes active and (b) one node active. The components $x^{(1)},\dots,x^{(5)}$ are represented by the colours blue, red, yellow, purple and green, respectively. Both panels have $r=2$. Panel (a) has $\gamma=3.04$, panel (b) has $\gamma=6.24$. The cycle shown in (a) is fragmentarily asymptotically stable, but the one in (b) is not, as can be seen at $i\approx 850$ where $x_5$ (green) becomes $O(1)$ and $x_2$ (red) remains on, indicating that the trajectory has moved away from the cycle with only one node active at any one time.}
\end{figure}

\section{Enumeration of fixed points and heteroclinic connections for a general directed graph}
\label{sec:general}

In this section, we describe how to find the fixed points and heteroclinic connections in phase space for any directed graph with inhibitory coupling. In section~\ref{sec:n1} that follows, we apply this to an $n$-node graph with nearest-neighbour coupling. 

\subsection{Enumeration of fixed points}
\label{sec:genfp}

For convenience and readability, we restate the general system~\eqref{eq:gen_coupled} with $\nnodes$ nodes:
\[
x_{i+1}^{(k)} = f(x_i^{(k)}) e^{-\gamma \sum_{\khat} A_{\khat k} x_i^{(\khat)}},\quad k=1,\dots, \nnodes
\]
where $A_{\khat k}$ is an adjacency matrix.
We enumerate the fixed points that can occur in this system, specifically, those with one or more non-zero coordinates. Fixed points can have any number of non-zero coordinates, so long as the corresponding nodes are not adjacent in the physical network. More formally, consider a partition of the first $\nnodes$ natural numbers into two sets:
\[
Z_+=\{\alpha_1,\alpha_2,\dots,\alpha_\jnodes\},\quad Z_0=\{\beta_1,\beta_2,\dots,\beta_{\nnodes-\jnodes}\},
\]
with $\jnodes<\nnodes$, $\alpha_k,\beta_k\in\{1,\dots,\nnodes\}$.
Then the point with
\begin{align*}
x^{(\alpha)}&=\xhat,\quad \alpha\in Z_+, \\
x^{(\beta)}&=0,\quad \beta\in Z_0,
\end{align*}
is a fixed point of~\eqref{eq:gen_coupled} if $A_{\hat{\alpha}\alpha}=0$ for all pairs $(\hat{\alpha},\alpha)\in Z_+\times Z_+$.  We label this fixed point $\xi_{Z_+}$. In the language of graph theory, $Z_+$ is called an \emph{independent set}.


\subsection{Existence of heteroclinic connections in phase space}
\label{sec:hetconns}

Consider a fixed point $\xi_{Z_+}$ in the system~\eqref{eq:gen_coupled}, with $|Z_+|=\jnodes$ (i.e. there are $\jnodes$ nodes active). We label the set of \emph{suppressed} nodes at $\xi_{Z_+}$ to be $Z_s(Z_+)$, where
\[
Z_s(Z_+)=\{ a_1, a_2, \dots ,a_{s(Z_+)} \}
\]
where for each $a_l$, there exists at least one $\alpha_k\in Z_+$ with $A_{a_l \alpha_k}=1$. We further define the \emph{suppression number} of a node $a_l$ to be the number of different $\alpha_k\in Z_+$ with $A_{a_l \alpha_k}=1$. The remaining nodes are the \emph{growing} nodes, and we define $Z_g(Z_+)=\{1,\dots, \nnodes\} \setminus (Z_+ \cup Z_s(Z_+))$.

It is simple to check that the linearisation of system~\eqref{eq:gen_coupled} about $\xi_{Z_+}$ has the following eigenvalues:
\begin{itemize}
\item $\jnodes$ eigenvalues equal to $2-r$, with eigenvectors in each of the directions corresponding to the active nodes.
\item $s(Z_+)$ eigenvalues in the suppressed directions. Each of these will be equal to $r e^{-n_s \gamma \hat{x}}$, where $n_s$ is the suppression number of that node.
\item $\nnodes-\jnodes-s(Z_+)$ eigenvalues equal to $r$. These are the \emph{growing} nodes. 
\end{itemize}

We assume that $1<r<3$, and $r e^{-\gamma \hat{x}}<1$, so each fixed point is a saddle. There are two ways in which heteroclinic connections between fixed points can arise. 

Consider a fixed point $\xi_{Z_+}$, and let $b\in\{1,\dots, \nnodes\} \setminus Z_+$. Consider the subspace in all components in $Z_+\cup \{b\}$ are fixed at zero. There are then three possible cases:
\begin{enumerate}
\item $b\in Z_s(Z_+)$, and so $\xi_{Z_+}$ is a sink in this subspace. There are no heteroclinic connections from $\xi_{Z_+}$ in this subspace.
\item $b\in Z_g(Z_+)$, and  $Z_+\cap Z_s(\{b\})=\emptyset$. Then there is a heteroclinic connection from $\xi_{Z_+}$ to $\xi_{Z_+\cup \{b\}}$. This is a heteroclinic connection of type $|Z_+|\rightarrow|Z_+|+1$.
\item  $b\in Z_g(Z_+)$, and $Z_+\cap Z_s(\{b\})=Z_-$ is non-empty. Then initial conditions near $\xi_{Z_+}$ will having an increasing $x^{(b)}$ component. All $x^{(a)}$ (for $a\in Z_-$) will eventually decay, and the trajectory will asymptote towards the fixed point $\xi_{Z_+\cup \{b\} \setminus \{Z_-\}}$ (which has node $b$ active but nodes in $Z_-$ inactive). This is a heteroclinic connection of type $|Z_+|\rightarrow|Z_+|-|Z_-|+1$.

\end{enumerate}

Note that for the directed graph with nearest neighbour coupling, each node only inhibits one single other node, and so in case 3 above,  $|Z_-|=1$ always. Thus for those examples, 
 the number of active nodes can increase via way of heteroclinic connections, but it can never decrease.


\section{Directed graph with  nearest-neighbour coupling}
\label{sec:n1}

In this section we consider the case of a general $n$-node ring graph, with one-way nearest-neighbour coupling. That is, the system~\eqref{eq:gen_coupled} with $A$ a cyclic permutation matrix, given by equations
\begin{equation}\label{eq:nnc}
x_{i+1}^{(k)} = f(x_i^{(k)}) e^{-\gamma x_i^{(k-1)}},\quad k=1,\dots, \nnodes.
\end{equation}
We refer to these graphs as $(\nnodes,1)$-graphs.

Note that the system~\eqref{eq:nnc} is equivariant with respect to the group $\mathbb{Z}_\nnodes$, generated by the element
$\sigma$, which has action
\begin{equation}\label{eq:sym}
\sigma(x^{(1)},\dots,x^{(\nnodes)})=(x^{(\nnodes)},x^{(1)},\dots,x^{(\nnodes-1)})
\end{equation}

As for the five-node system, equations~\eqref{eq:nnc} have $n$ fixed point solutions $\xi_j$ each with the $j$th component equal to $\hat{x}$, and all other components zero. Other fixed points are found using the method in section~\ref{sec:genfp}\footnote{For the $(N,1)$-graph case the fixed points can be enumerated directly with the Lucas numbers $L_n$, given by $L_1 = 1, L_2 = 3, L_{n} = L_{n-1}+L_{n-2}, n\ge 3$. The $(N,1)$-graph has $L_n$ fixed points, including the configuration with $N$ zero components.}. Note that for any set of nodes $Z_+$, $Z_s(Z_+)=|Z_+|$, that is, the number of suppressed nodes is the same as the number of active nodes, and each suppressed node has a suppression number of one.

\subsection{Examples of heteroclinic networks}

Using the methods described in section~\ref{sec:general} we give some further examples of the heteroclinic networks which occur in $(\nnodes,1)$-graphs.

\subsubsection{$(6,1)$-graph}

The $(6,1)$-graph has six fixed points with one non-zero component, nine with two non-zero components, and two with three non-zero components. The latter are stable and the remainder are saddles. Each of the $\xi_j$ fixed points with one non-zero component has heteroclinic connections to $\xi_{j-1}$, and to the three fixed points $\xi_{j,j+2}$, $\xi_{j,j+3}$ and $\xi_{j,j+4}$ with two non-zero components. Note that the fixed points with two non-zero components can be divided into two classes depending on whether the spaces between the active components is $2$ and $2$ (the fixed points $\xi_{j,j+3}$) or $1$ and $3$ (the fixed points  $\xi_{j,j+2}$). The subset of the heteroclinic network between just these fixed points is shown in figure~\ref{fig:six1} .

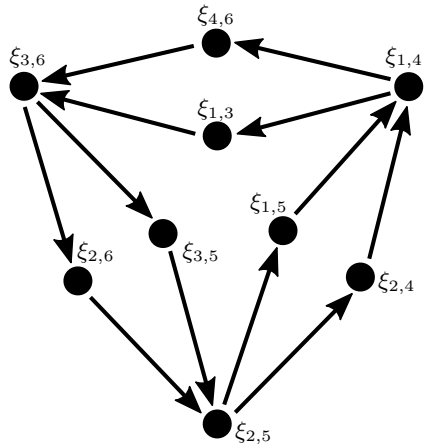
\begin{figure}
\begin{center}
\setlength{\unitlength}{1mm}
\begin{picture}(55,55)(0,0)
\put(0,0){\includegraphics[width=5.5cm]{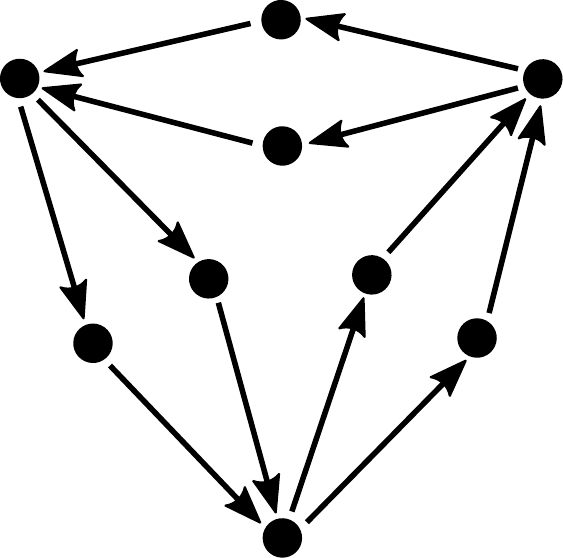}}

\put(50,50){$\xi_{1,4}$}
\put(0,50){$\xi_{3,6}$}
\put(25,55.5){$\xi_{4,6}$}
\put(49,20){$\xi_{2,4}$}
\put(9,24){$\xi_{2,6}$}
\put(30,0){$\xi_{2,5}$}
\put(23,24){$\xi_{3,5}$}
\put(25,43){$\xi_{1,3}$}
\put(32,30.5){$\xi_{1,5}$}

\end{picture}
\end{center}
\caption{\label{fig:six1} The figure shows part of the resulting heteroclinic network between fixed points for the $(6,1)$-graph. Not shown are the $\xi_j$ fixed points, the heteroclinic connections from these, or the $x_{j,j+2,j+4}$ fixed points (which are stable).
}
\end{figure}

\subsubsection{$(7,1)$-graph}

The $(7,1)$-graph has seven fixed points with one non-zero component, fourteen with two non-zero components, and seven with three non-zero components. The fixed points with two non-zero components can further be divided into two classes of types $\xi_{j,j+2}$ and $\xi_{j,j+3}$. The heteroclinic network between the fixed points with two or three non-zero components is shown in figure~\ref{fig:seven1}. Each of the $\xi_j$ fixed points, which is not shown here, will have heteroclinic connections to $\xi_{j-1}$, $\xi_{j,j+2}$, $\xi_{j,j+3}$ and $\xi_{j,j+4}$.

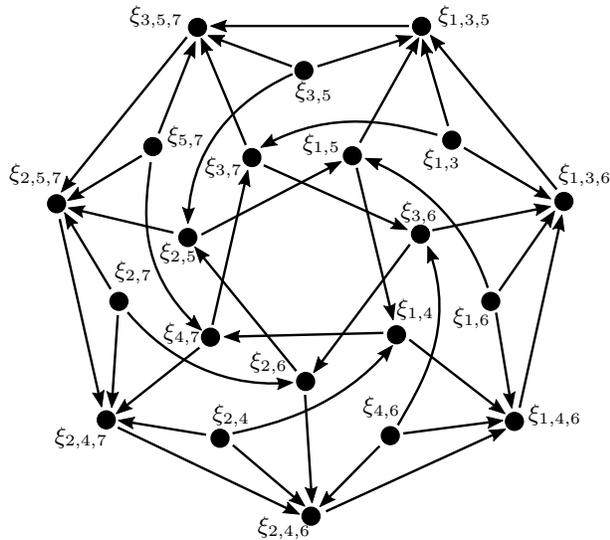
\begin{figure}
\begin{center}
\setlength{\unitlength}{1mm}
\begin{picture}(75,75)(-5,0)
\put(0,0){\includegraphics[width=7cm]{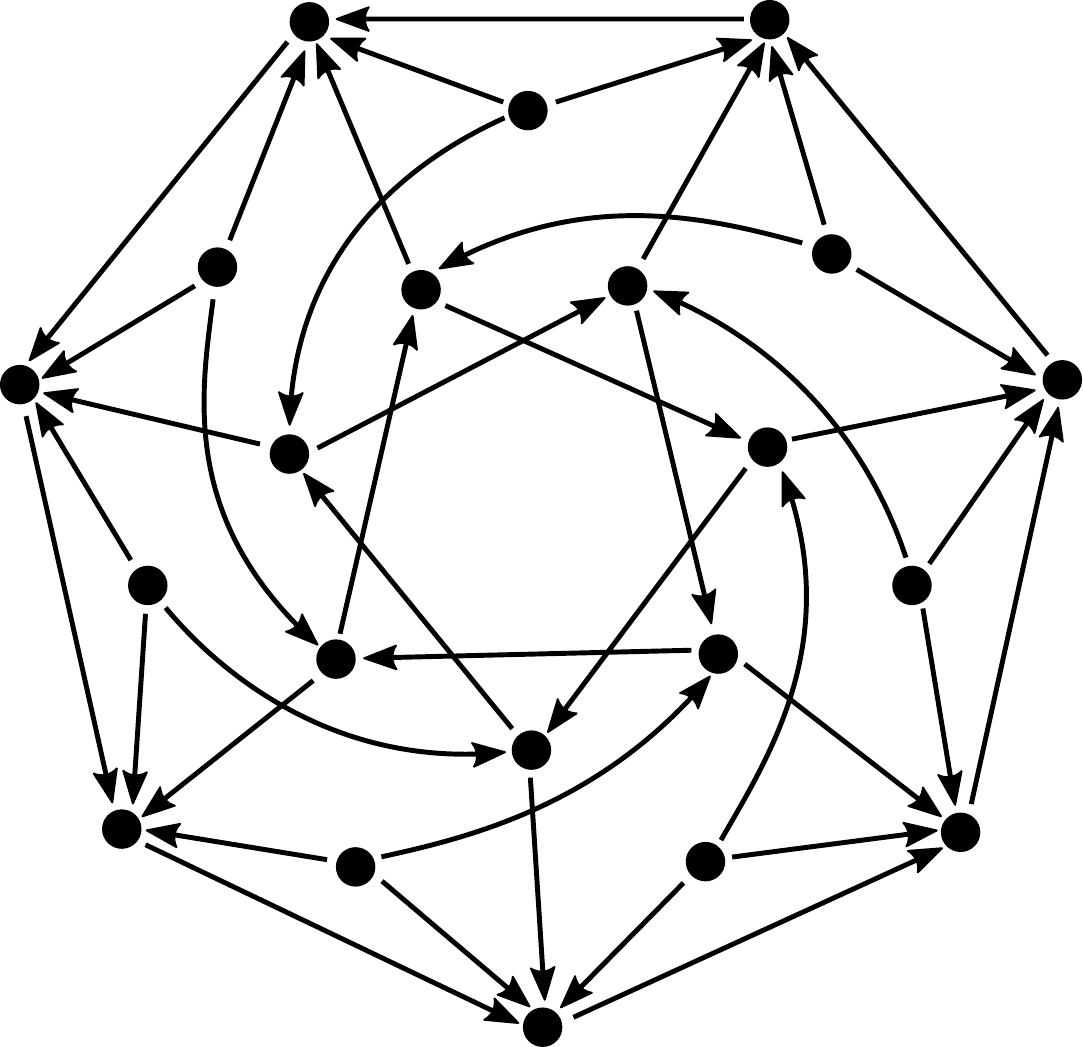}}

\put(52,67){$\xi_{1,3,5}$}
\put(11,67){$\xi_{3,5,7}$}
\put(-5,46){$\xi_{2,5,7}$}
\put(68,46){$\xi_{1,3,6}$}
\put(1,11){$\xi_{2,4,7}$}
\put(64,14){$\xi_{1,4,6}$}
\put(28,-1){$\xi_{2,4,6}$}

\put(50,48){$\xi_{1,3}$}
\put(16,51){$\xi_{5,7}$}
\put(33,57){$\xi_{3,5}$}
\put(54,27){$\xi_{1,6}$}
\put(9,33){$\xi_{2,7}$}
\put(42,15){$\xi_{4,6}$}
\put(22,14){$\xi_{2,4}$}

\put(21,47){$\xi_{3,7}$}
\put(34,50){$\xi_{1,5}$}
\put(47,41){$\xi_{3,6}$}
\put(46.5,28){$\xi_{1,4}$}
\put(27,21){$\xi_{2,6}$}
\put(15,35.5){$\xi_{2,5}$}
\put(15.5,24.8){$\xi_{4,7}$}

\end{picture}
\end{center}
\caption{\label{fig:seven1} The figure shows part of the resulting heteroclinic network between fixed points for the $(7,1)$-graph. Not shown are the $\xi_j$ fixed points and the heteroclinic connections from these.
}
\end{figure}

\subsection{Symmetric subcycles}

As can be seen from the examples given so far of the $(5,1)$-, $(6,1)$- and $(7,1)$-graphs, there can exist many different heteroclinic cycles within the heteroclinic network in phase space. As the number of nodes in the ring graph increases, so too will the number of heteroclinic cycles. In the following, we establish stability results for some of these subcycles. Although in theory our method can be used for any subcycle, in practise it is much easier to compute explicit stability conditions if the cycle is \emph{symmetric}. That is, if the cycle is between $n$ fixed points $\zeta_1,\dots,\zeta_n$, then there exist a symmetry $\rho=\sigma^M$ (for some $M$) such that $\zeta_j=\rho\zeta_{j-1}$ for all $j\in 2,\dots,n$ (and  $\zeta_1=\rho\zeta_n$).

Note that for both the $(6,1)$ and the $(7,1)$-graph, such cycles exist between the $\xi_j$ fixed points. For the $(7,1)$-graph, symmetric cycles also exist between the $\xi_{j,j+3}$ fixed points and the $\xi_{j,j+2,j+4}$ fixed points. However, for the $(6,1)$-graph, there is no symmetric cycle between fixed points with two non-zero components (see again figure~\ref{fig:six1}).

In this section we enumerate the possible symmetric cycles in $(\nnodes,1)$-graphs. Clearly, the symmetry requires the number of active coordinates at each fixed point to be the same, and additionally, that the spacing of the active coordinates around the ring is the same at each fixed point. This gives restrictions on the allowed spacing between the on nodes, as follows.

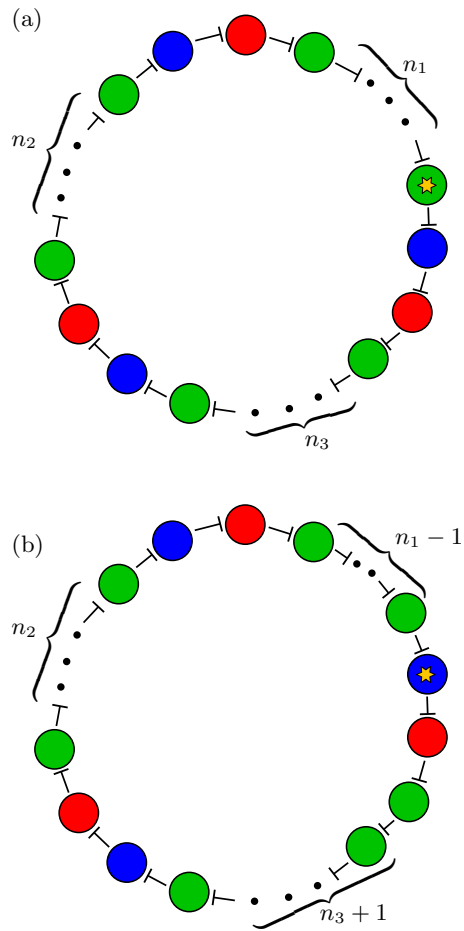
\begin{figure}
\setlength{\unitlength}{1mm}
\begin{center}
\begin{picture}(55,122)(-3,-5)

\put(-2,103){\rotatebox{-20}{$\left\{\color{white}{\begin{pmatrix} 0 \\ 0 \\ 0 \\ 0 \\ 0\end{pmatrix}} \right.$}}
\put(37,114){\rotatebox{223}{$\left\{\color{white}{\begin{pmatrix} 0 \\ 0 \\ 0 \\ 0\end{pmatrix}} \right.$}}
\put(26,66){\rotatebox{105}{$\left\{\color{white}{\begin{pmatrix} 0 \\ 0 \\ 0 \\ 0\end{pmatrix}} \right.$}}

\put(-3,104){$n_2$}
\put(36,64){$n_3$}
\put(49,114){$n_1$}

\put(0,67){\includegraphics[width=55mm]{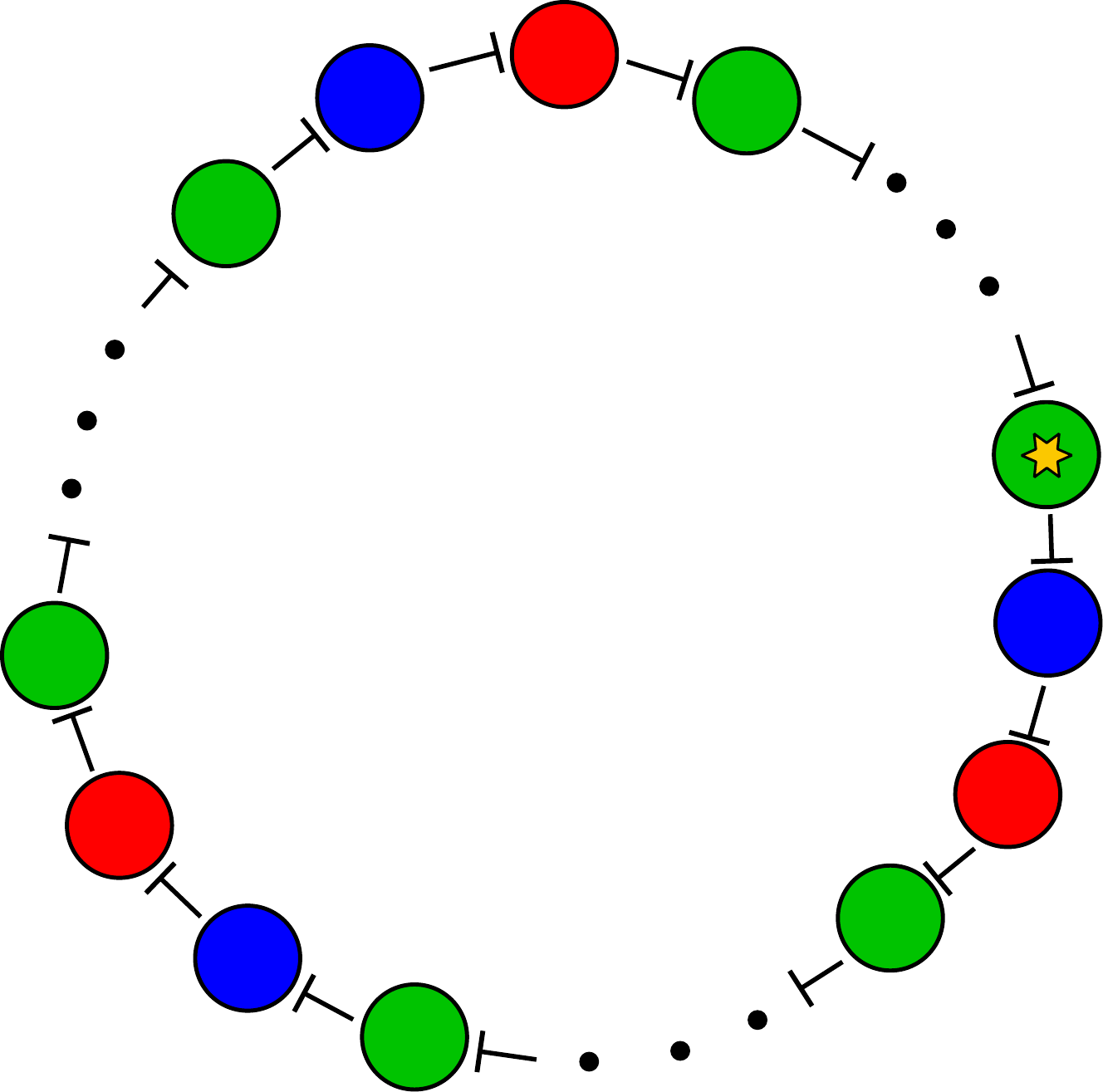}}

\put(-2,38){\rotatebox{-20}{$\left\{\color{white}{\begin{pmatrix} 0 \\ 0 \\ 0 \\ 0 \\ 0\end{pmatrix}} \right.$}}
\put(35,51){\rotatebox{230}{$\left\{\color{white}{\begin{pmatrix} 0 \\ 0 \\ 0 \\ 0\end{pmatrix}} \right.$}}
\put(25,3){\rotatebox{115}{$\left\{\color{white}{\begin{pmatrix} 0 \\ 0 \\ 0 \\ 0 \\ 0 \\ 0\end{pmatrix}} \right.$}}

\put(-3,39){$n_2$}
\put(38,1){$n_3+1$}
\put(48,51){$n_1-1$}

\put(0,2){\includegraphics[width=55mm]{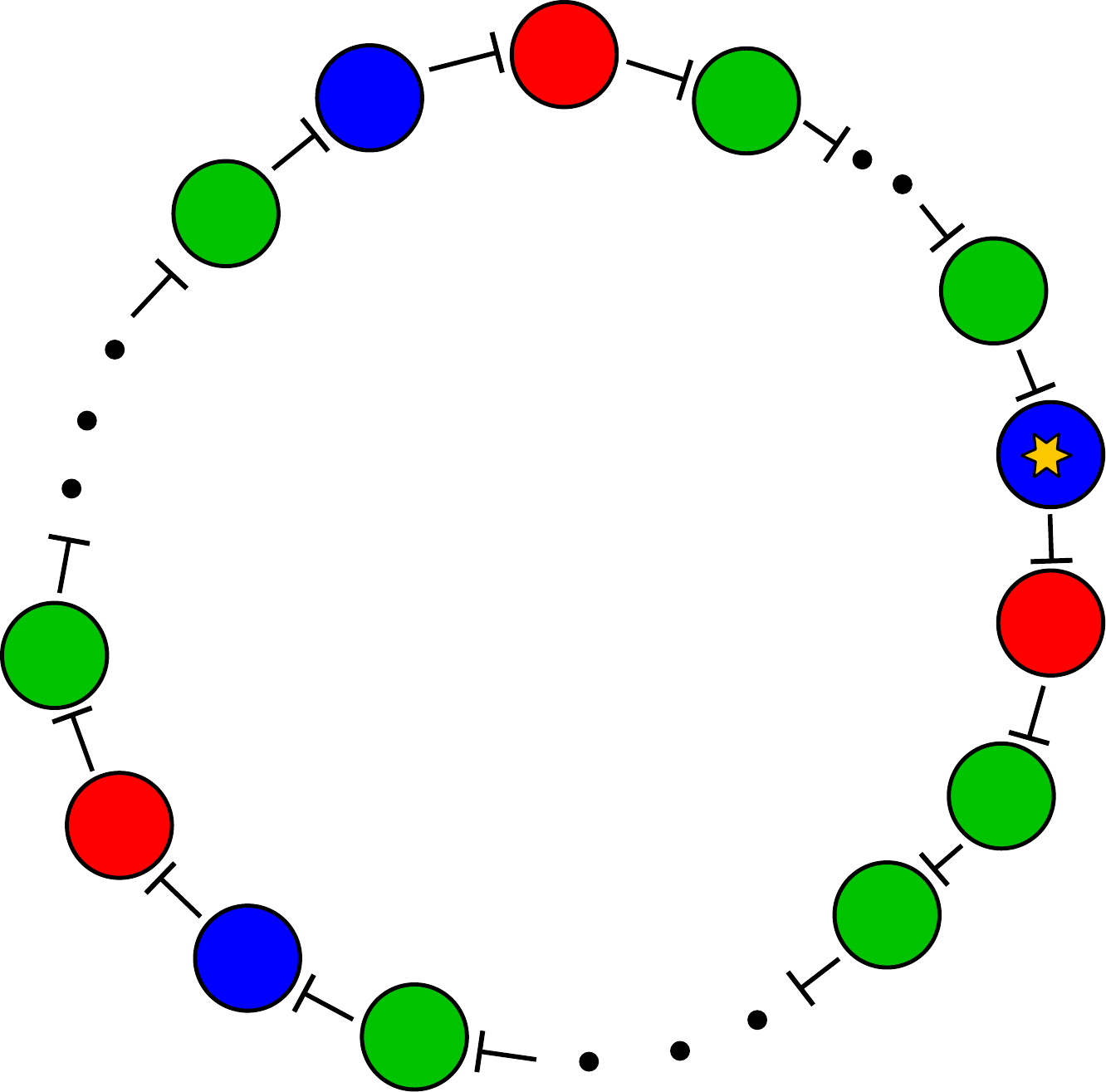}}

\put(-3,120){(a)}
\put(-3,50){(b)}

\end{picture}
\end{center}

\caption{\label{fig:spacing} The figure shows schematic diagrams of two fixed points in an $\nnodes$-ring graph with nearest-neighbour one-way connections. The blue nodes are active, the red nodes are supressed and the remainder are green. The fixed points are considered to be adjacent in a heteroclinic cycle, and the numbers (e.g.~$n_1$) indicate the number of green nodes between the red/blue pairs.}
\end{figure}

In figure~\ref{fig:spacing}(a),  we show an $n$-node ring graph, and suppose that there are three nodes which are active (coloured blue). The nodes coloured red are those which are being supressed by the blue nodes, and all others are coloured green. Suppose that there are $n_1$, $n_2$ and $n_3$ green nodes between each pair of blue and red nodes (as marked in the figure) (where $n_1\geq 1$, $n_2,n_3\geq 0$). In order for the next fixed point to have the same number of active components, the next node to reach $O(1)$ much be adjacent to a blue node: the heteroclinic connection must of type $3\rightarrow 3$.  Without loss of generality, we suppose that the next node to reach $O(1)$ is the node marked with a yellow star. The next fixed point in the cycle thus has nodes coloured as in figure~\ref{fig:spacing}(b) (the yellow star marks the same node). There are then two options for how we could rotate the arrangement at the second fixed point to match the arrangement at the first. Either (i) we rotate panel (b) anticlockwise by $n_1+1$ nodes, and  must therefore have
\[
n_1-1=n_2,\quad n_2=n_3, \quad n_3+1=n_1
\]
which implies $n_2=n_3=n_1-1$;
or, (ii), we rotate panel (b) clockwise by $n_1+1$ nodes, and then we must have
\[
n_1-1=n_3, \quad n_3+1=n_2, \quad n_2=n_1
\]
which implies $n_1=n_2=n_3+1$.

If there are instead $\jnodes>3$ nodes active, rather than just three, similar arguments can be made, and the results give the same two possible cases.
 In case (i), we write $n_1=p-1$, $n_2,\dots, n_\jnodes=p-2$, and the total number of nodes is $\nnodes=p\jnodes+1$, for some $p\geq 2$. In case (ii), we write $n_\jnodes=s-3$, $n_1,\dots,n_{\jnodes-1}=s-2$ and the total number of nodes is $\nnodes=s\jnodes-1$, for some $s\geq 3$.

In the case where there is only a single active node ($\jnodes=1$), then clearly all fixed points are symmetric. In the case where $\jnodes=2$, the same arguments apply as for $\jnodes=3$ or more, except that there is no distinction between cases (i) and (ii) above.

Note that in both cases (i) and (ii), the spacing between the active nodes is such that all gaps between active nodes are of equal length except one that is one greater or fewer than the others. In figure~\ref{fig:elevens} we show four possible types of fixed points in the graph with $\nnodes=11$, with two, three, four and five active nodes respectively. The fixed points with three (figure~\ref{fig:elevens}(b)) and four (figure~\ref{fig:elevens}(c)) active nodes  are in case (ii): one gap is smaller than the others. The fixed point with five active nodes  (figure~\ref{fig:elevens}(d)) is in case (i): one gap is bigger than the others.

The maximum value that $\jnodes$ can take (the maximum number of active nodes at a fixed point) is equal to $\nnodes/2$ if $\nnodes$ is even, and $(\nnodes-1)/2$ if $\nnodes$ is odd. We refer to these fixed points as \emph{maximally active} fixed points.

\begin{figure}
\setlength{\unitlength}{1mm}
\begin{center}
\begin{picture}(75,75)(0,0)
\put(0,40){\includegraphics[width=35mm]{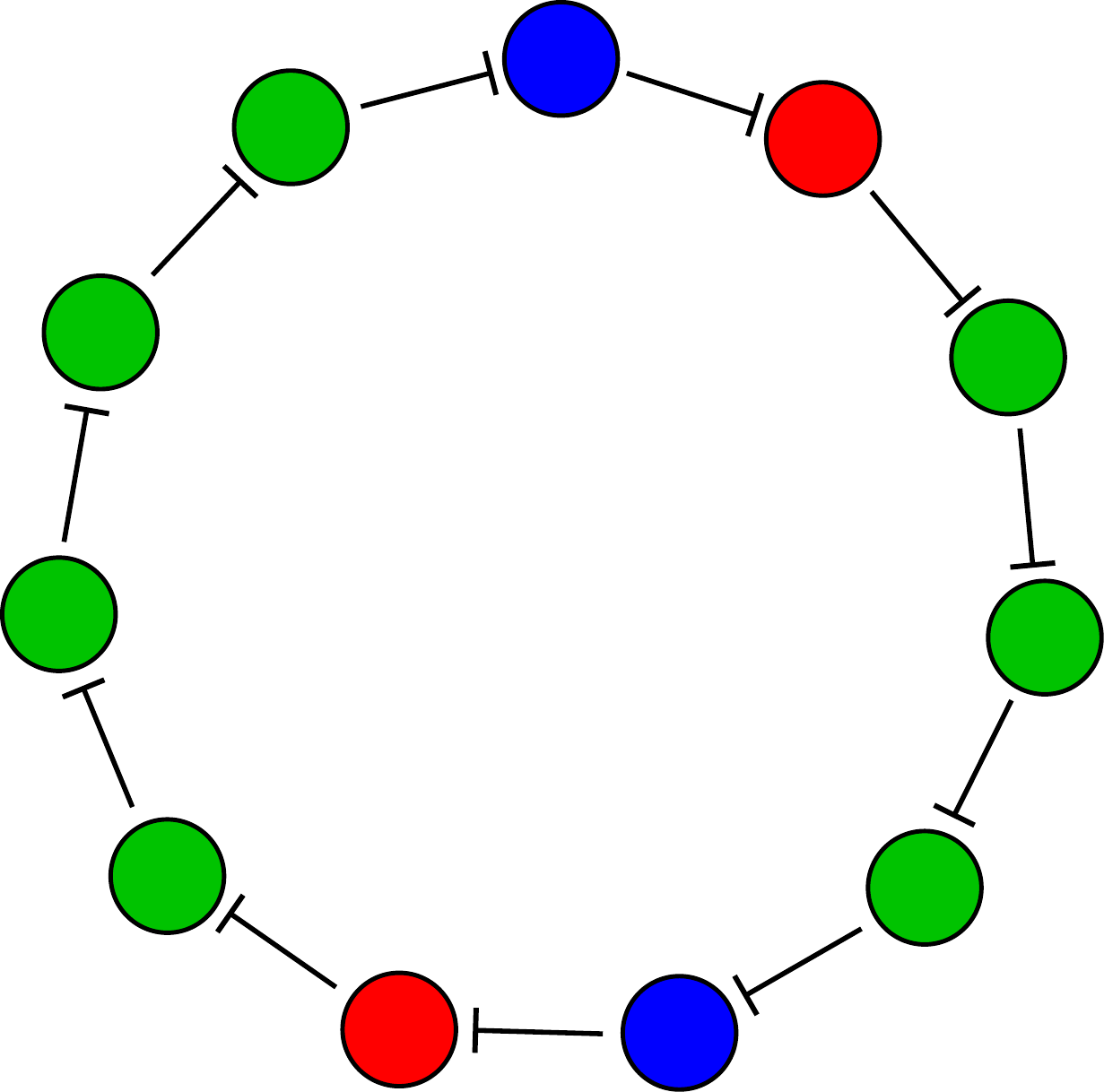}}
\put(40,40){\includegraphics[width=35mm]{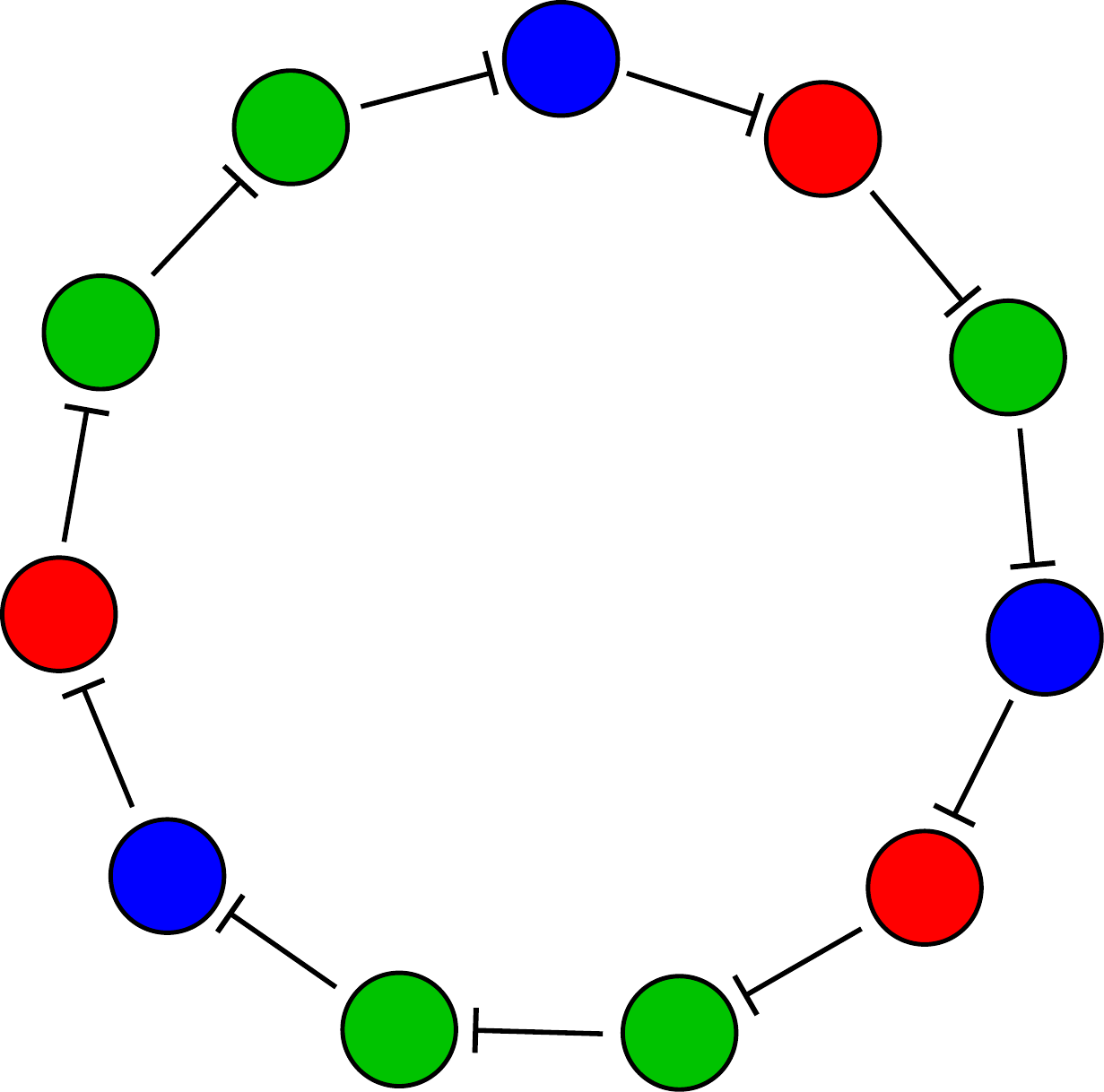}}

\put(0,0){\includegraphics[width=35mm]{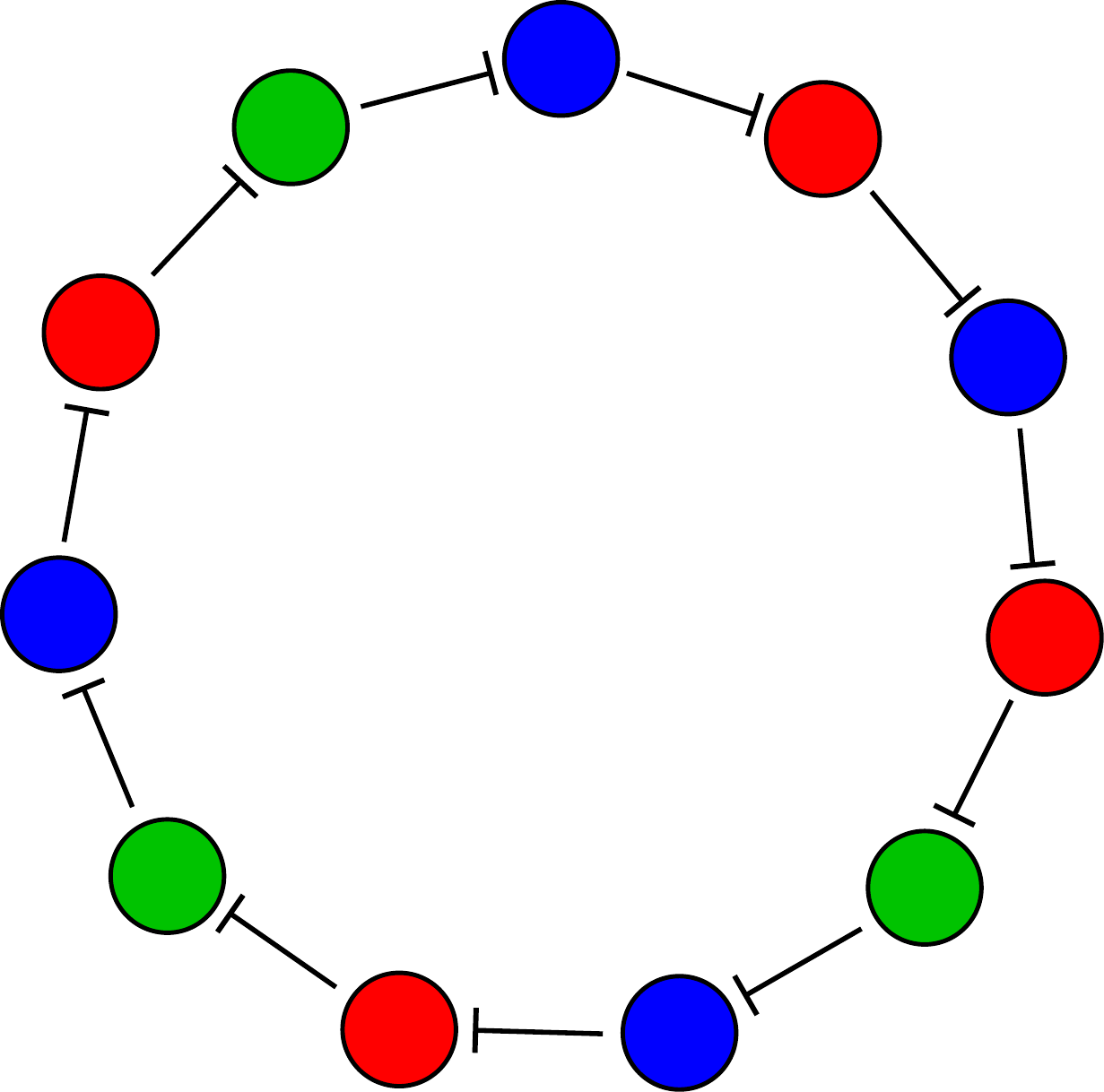}}
\put(40,0){\includegraphics[width=35mm]{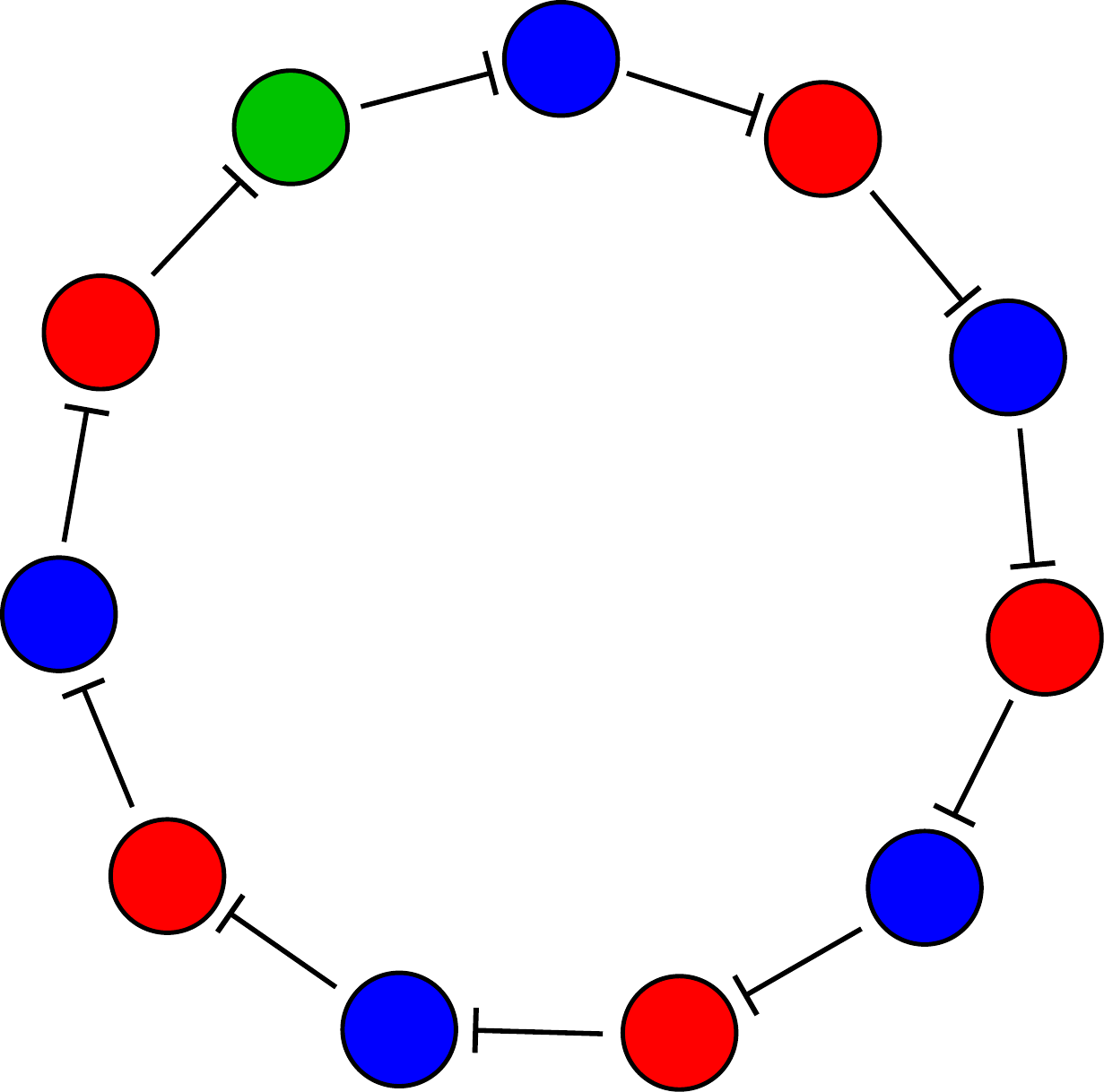}}

\put(0,72){(a)}
\put(40,72){(b)}
\put(0,32){(c)}
\put(40,32){(d)}

\end{picture}
\end{center}
\caption{\label{fig:elevens} The figures show fixed points in an eleven node network with (a) two, (b) three, (c) four, and (d) five active nodes respectively. Each of these fixed points has a spacing between the active nodes where all gaps are equal except one, that is one greater or fewer than the others.}
\end{figure}
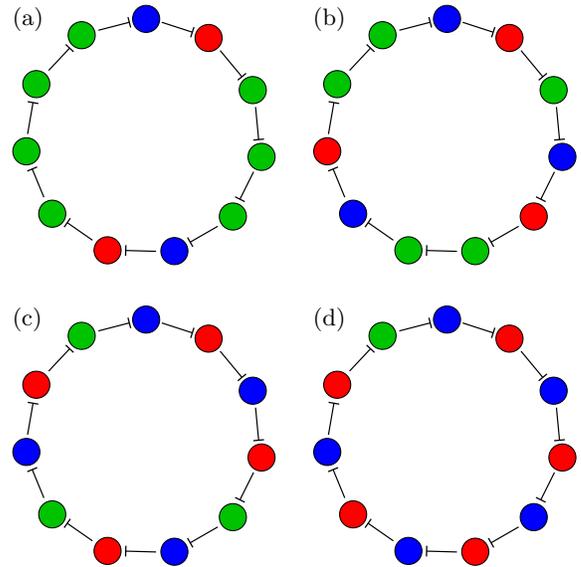

\subsection{Construction and analysis of heteroclinic cycles}

We now specifically construct a heteroclinic connection between two fixed points, and use this construction to show how the stability of a heteroclinic cycle between fixed points can be computed. We first consider the dynamics within an \emph{epoch}, which we define to be the length of time a trajectory spends in a neighbourhood of a single fixed point. We then discuss how the trajectory transitions between epochs. This method echos the construction of Poincar\'e maps which is typical in analysis of heteroclinic cycles in continuous-time systems~\cite{kirk1994competition,krupa2004asymptotic,kirk2012resonance,podvigina2012stability}.

\subsubsection{Dynamics within one epoch}

\begin{figure}
\setlength{\unitlength}{1mm}
\begin{center}
\begin{picture}(60,80)(0,0)

\put(11,0){\includegraphics[width=35mm]{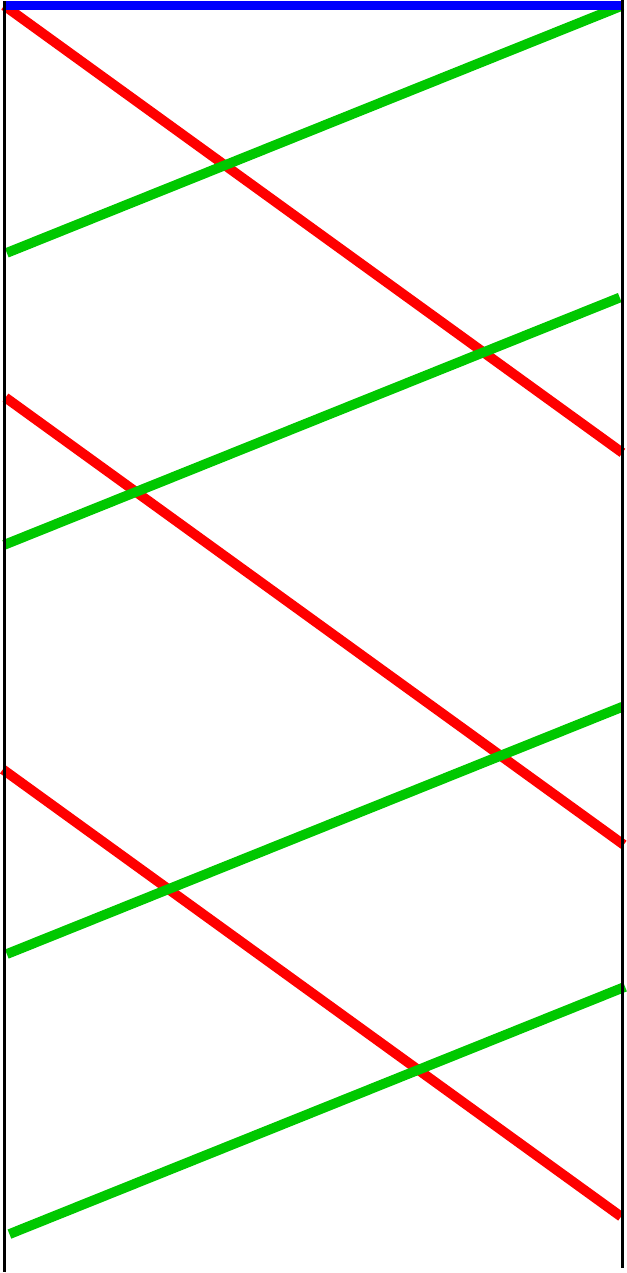}}

\put(3,70){$O(1)$}

\put(0,56){$\log x_e^{\iin}$}
\put(-2,49){$\log x_{t_{\jnodes-1}}^{\iin}$}
\put(0,40){$\log x_{s_1}^{\iin}$}
\put(0,28){$\log x_{t_{1}}^{\iin}$}
\put(0,18){$\log x_{s_2}^{\iin}$}
\put(0,1){$\log x_{s_\lnodes}^{\iin}$}

\put(48,54){$\log x_{s_1}^{\out}$}
\put(48,45){$\log x_c^{\out}$}
\put(48,32){$\log x_{s_2}^{\out}$}
\put(48,23){$\log x_{t_{\jnodes-1}}^{\out}$}
\put(48,15){$\log x_{s_\lnodes}^{\out}$}
\put(48,3){$\log x_{t_1}^{\out}$}


\put(30,0){\vector(1,0){10}}
\put(35,-2.5){$i$}

\end{picture}
\end{center}

\caption{\label{fig:schem_ts} The figure shows a schematic of a time series of one epoch of a trajectory close to a fixed point in system~\eqref{eq:nnc}.}
\end{figure}
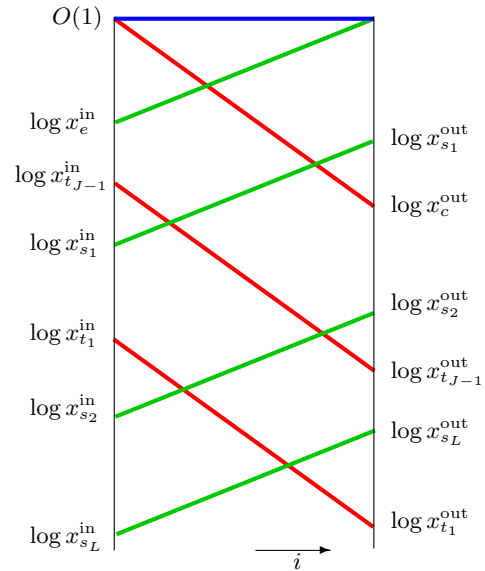

Consider the dynamics of~\eqref{eq:nnc}, with $\nnodes$ nodes in the graph, near a fixed point at which $\jnodes$ nodes are active, that is, an equilibrium at which $\jnodes$ components are equal to $\hat{x}$. Then, by the arguments given in section~\ref{sec:hetconns}, there will be $\jnodes$ negative eigenvalues, and hence $\jnodes$ components which are decaying. Similarly, there will be $\nnodes-2\jnodes$ components which are growing. We give a schematic sketch of this in figure~\ref{fig:schem_ts}, and note that the only initial condition we have specified is that one of the decaying components starts at $O(1)$. We have further specified that one of the growing components reaches $O(1)$ at the end of the period of time shown in the figure. We label this period of time one \emph{epoch}, and note that after this epoch, the trajectory will be in the neighbourhood of a different fixed point, which may have the same, or one more, nodes that are active. 

We label each of the growing components, in order of largest to smallest initial conditions:
\[
x_e, x_{s_1},\dots, x_{s_\lnodes},
\]
where $\lnodes=\nnodes-2\jnodes-1$, and we label each of the decaying components, again in order of largest to smallest initial conditions:
\[
x_c, x_{t_{\jnodes-1}},x_{t_{\jnodes-2}},\dots,x_{t_1}.
\]
We use a superscript `in' and `out' to indicate the initial, and final, conditions for each of the components. Recall that each of the growing components grows at a rate $r>1$, and each of the decaying components decays at a rate $r\exp(-\gamma \hat{x})<1$. Let the number of iterations in the epoch shown in figure~\ref{fig:schem_ts} be $T$, and then we have
\[
\log(x_e^{\iin})+T\log r=O(1),
\]
or, assuming that $x_e^{\iin}\ll1$ and hence $T$ is large
\[
T=-\frac{\log(x_e^{\iin})}{\log r}+O(1).
\]
We can then use this expression for $T$ to compute the `out' coordinates of the other components in terms of the `in' components. Specifically, we find
\begin{align*}
\log(x_{s_i}^{\out})&=\log(x_{s_i}^{\iin})-\log(x_e^{\iin}),\quad  i=1,\dots,\lnodes,\\
\log(x_{t_i}^{\out})&=\log(x_{t_i}^{\iin})+\left(\frac{\gamma \hat{x}}{\log r}-1\right) \log(x_e^{\iin}),\quad  i=1,\dots,\jnodes-1, \\
\log(x_c^{\out})&=\left(\frac{\gamma \hat{x}}{\log r}-1\right) \log(x_e^{\iin}). 
\end{align*}

We write $\delta=\left(\frac{\gamma \hat{x}}{\log r}-1\right)$, and $X_c=\log(x_c)$ etc, and then have the following linear map from the `in' variables to the `out' variables:
\begin{equation}\label{eq:linmap1}
\begin{pmatrix}
X_{s_i}^{\out} \\
X_{t_i}^{\out} \\
X_c^{\out} 
\end{pmatrix}
=
\begin{pmatrix}
-1_{l} & I_{l} & 0 \\
\delta 1_{j-1} & 0 & I_{j-1} \\
\delta & 0 & 0
\end{pmatrix}
\begin{pmatrix}
X_{e}^{\iin} \\
X_{s_i}^{\iin} \\
X_{t_i}^{\iin}
\end{pmatrix}
\end{equation}
where $1_{m}$ is a length $m$ column vector of $1$'s, and $I_m$ is the $m\times m$ identity matrix.

\subsubsection{Transitions between epochs}

In this section, we discuss how to  map the `out' variables given in equation~\eqref{eq:linmap1} onto a new set of `in' variables for the next epoch.




In figure~\ref{fig:casea} we show how the nodes are labelled (as in the time series schematic shown in figure~\ref{fig:schem_ts}), for case (i). The node with the yellow star is $x_e$, as described above. The remaining growing nodes are $x_{s_1},\dots, x_{s_\lnodes}$, where $\lnodes=\nnodes-2\jnodes-1$. As per the labelling scheme in figure~\ref{fig:schem_ts}, $x_{s_1}$ is the component which will become $O(1)$ next in the sequence (following $x_e$), and by applying the necessary rotations between fixed points, it can be seen how this node is selected. Similar arguments explain how the other growing nodes are labelled. The labelling of the contracting nodes is done in a similar fashion: $x_c$ is the component which was $O(1)$ at the start of the epoch, and hence was blue in the previous fixed point: application of the rotation between panels (a) and (b) in figure~\ref{fig:spacing} gives us this label. The $x_{t_i}$ labels are found in a similar way.

This results, for case (i), in the following transformation between the `out' coordinates of the last fixed point, and the `in' coordinates of the next one:
\begin{align*}
x_{s_1}^{\out}& = x_e^{\iin} \\
x_{s_i}^{\out}& = x_{s_{i-1}}^{\iin}, \quad i=2,\dots, \lnodes,\\
x_{t_1}^{\out}& = x_{s_l}^{\iin}, \\
x_{t_{i}}^{\out} & = x_{t_{i-1}}^{\iin}, \quad i=2,\dots, \jnodes \\
x_{c}^{\out}& = x_{t_\jnodes}^{\iin}.
\end{align*}

\begin{figure}
\setlength{\unitlength}{1mm}
\begin{center}
\begin{picture}(74,68)(-6,-5)



\put(-1,40){\rotatebox{-20}{$\left\{\color{white}{\begin{pmatrix} 0 \\ 0 \\ 0\end{pmatrix}} \right.$}}
\put(48,53){\rotatebox{220}{$\left\{\color{white}{\begin{pmatrix} 0 \\ 0 \\ 0\end{pmatrix}} \right.$}}
\put(33,-1){\rotatebox{105}{$\left\{\color{white}{\begin{pmatrix} 0 \\ 0 \\ 0\end{pmatrix}} \right.$}}

\put(-6,41){$p-2$}
\put(41,-3){$p-2$}
\put(60,53){$p-1$}

\put(0,0){\includegraphics[width=65mm]{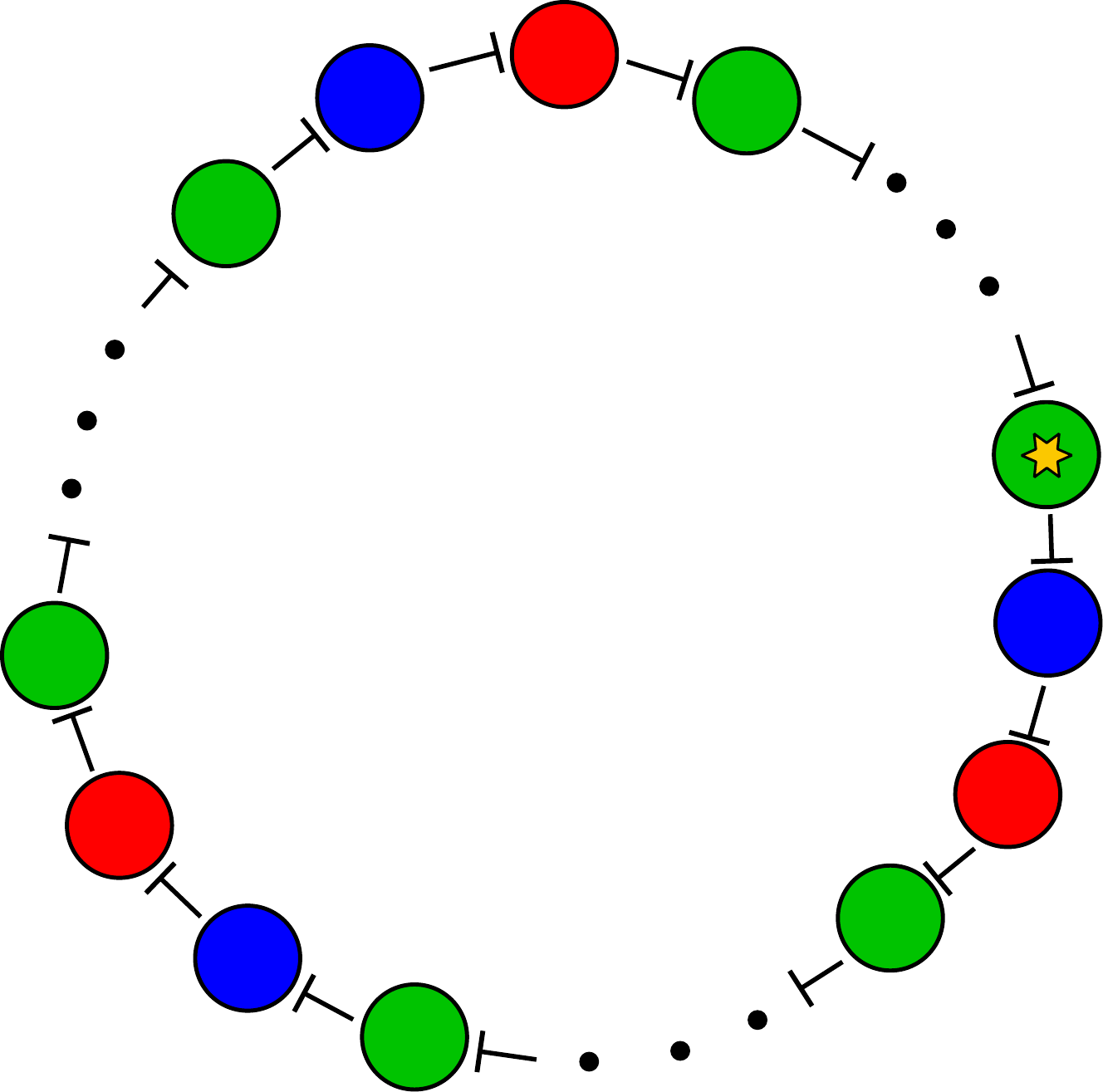}}

\put(55,35){$x_e$}
\put(32,55){$x_c$}
\put(52,20){$x_{t_1}$}
\put(11,18){$x_{t_\jnodes}$}
\put(23,7.5){$x_{s_1}$}
\put(17,48){$x_{s_2}$}
\put(43,52.5){$x_{s_{\lnodes}}$}



\end{picture}
\end{center}
\caption{\label{fig:casea} The figure shows labelling of the nodes in case (i).}
\end{figure}

Combing this with the linear map in~\eqref{eq:linmap1} we get that the logarithmic coordinates each epoch are transformed under the map
\begin{equation}
X\rightarrow MX \label{eq:linmap}
\end{equation}
where $M$ is called a  \emph{transition matrix} and is given by
\begin{equation}\label{eq:M}
M=\begin{pmatrix}
-1 & 1 & 0 & 0 & 0 & \dots & 0 & 0 \\
-1 & 0 & 1 & 0  & 0 & \dots & 0 & 0\\
\vdots & \vdots & \ddots & \ddots & \ddots  & &  & \vdots \\
-1 & 0 & \dots & 0 & 1 & 0 & \dots & 0\\
\delta & 0 & \dots & & 0 & 1& 0 & 0\\
\vdots & \vdots & &  & & \ddots & \ddots & \vdots \\
\delta & 0 & \dots & 0& 0& 0 & 0 & 1 \\
\delta & 0 & \dots & 0& 0 & 0 & 0& 0
\end{pmatrix}.
\end{equation}
The matrix $M$ is a $q\times q$-square matrix, where $q=\nnodes-\jnodes-1=\jnodes(p-1)$. There are $\lnodes=\nnodes-2\jnodes-1=\jnodes(p-2)$ rows starting with a $-1$ and $\jnodes$ rows starting with a $\delta$, and $1$'s on the upper diagonal. Here $\delta=\frac{\gamma \hat{x}}{\log r}-1$, as before.

In case (ii), the labelling of the coordinates can be computed in the same way (although the labelling turns out to be different), but the resulting map is exactly the same. That is,  $M$ is a  $q\times q$-square matrix, with $q=\nnodes-\jnodes-1=\jnodes(s-1)-2$, the number of  rows of $M$ starting with a $-1$ is $\lnodes=\nnodes-2\jnodes-1=\jnodes(s-2)-2$, the number of rows starting with a $\delta$ is still $\jnodes$. 

In the original $x$ coordinates, the map~\eqref{eq:linmap} has a fixed point at $x=0$, which corresponds to the heteroclinic cycle in the original system. 
Podvigina~\cite{podvigina2012stability} gives results on the stability of this heteroclinic cycle, dependent on properties of the eigenvalues and eigenvectors of the transition matrix. In the next section we give a brief heuristic argument explaining Podvigina's results, and then state the precise requirements for stability.

\subsection{Transition matrices and fragmentary asymptotic stability}

We begin this section with some formal definitions, referring back to a generic dynamical system of the form~\eqref{eq:genf}.

We define the  $\delta$-local basin of attraction of a set $H$, invariant under $f$,  as $\mathcal{B}_\delta(H)$:
\begin{equation}
\begin{split}
\mathcal{B}_\delta(H)=\{\vec{x}\in\R^n ~|~ | f^i(\vec{x}), H|<\delta\ \forall\ i\geq 0,\ \\ \mathrm{and}\ \lim_{i\rightarrow \infty} |f^i(\vec{x}),H|=0 \}.
\end{split}
\end{equation}

From Podvigina~\cite{podvigina2012stability}, we also have:
\begin{definition}\label{defn:fas}
An invariant set $H$ is \emph{fragmentarily asymptotically stable} if, for any $\delta>0$, 
\[
\mu(\mathcal{B}_\delta(X))>0,
\]
where $\mu$ is the Lebesgue measure of a set in $\R^n$.
\end{definition}

Now suppose that for a heteroclinic cycle $H$, we have derived a Poincar\'e map in logarithmic coordinates, as in the previous section, of the form
\[
X_{i+1}=M X_{i}
\]
Here the subscript index $i$ now counts \emph{epochs}, rather than individual iterations of the original map. Let $X_i$ have dimension $q$, so $M$ is an $q\times q$ matrix, and then we can write the initial condition $X_0$, in the basis of eigenvectors $v_j$ of $M$, i.e.
\[
X_0=\Sigma_{j=1}^{q} c_j v_j
\]
where the $c_j$ are scalars. Then we find
\[
X_i=\Sigma_{j=1}^{q} \lambda_j ^i c_j v_j
\]
where $\lambda_j$ is the eigenvalue corresponding to the eigenvector $v_j$. Let $\lambda_{\max}$ be the eigenvalue with largest absolute value, and then the leading order term of $X_i$ is
\[
X_i\approx  \lambda_{\max} ^i c_{\max} v_{\max}
\]
Recall that $X_i$ are logarithmic variables, so $x_i\rightarrow 0$ if $X_i\rightarrow -\infty$, that is, in order for the heteroclinic cycle to be stable, we require at least that $|\lambda_{\max}|>1$. However, in addition, the $X_i$ are required to stay real and negative, so additional conditions are required, namely that $\lambda_{\max}$ is real, and that all the entries in the eigenvector $v_{\max}$ are of the same sign. Podvigina shows that if these conditions are satisfied, then there exists an open set of initial conditions which remain close to the heteroclinic cycle for all time, more specifically, the heteroclinic cycle is \emph{fragmentarily asymptotically stable}.

\begin{lemma}[Adapted from Podvigina~\cite{podvigina2012stability}]\label{lemma:fas}
Let $M$ be a transition matrix for a heteroclinic cycle $H$. Let $\lambda_\mathrm{max}$ be the eigenvalue with largest absolute value of the matrix $M$, and $v_{\max}$ be the associated eigenvector. Suppose $\lambda_\mathrm{max}\neq 1$. Then $H$ is fragmentarily asymptotically stable if the following conditions hold:
\begin{enumerate}
\item $\lambda_\mathrm{max}$ is real
\item $\lambda_\mathrm{max}>1$
\item $v_{\mathrm{max}}^l v_{\mathrm{max}}^j>0$ for all $l,j$.
\end{enumerate}
\end{lemma}
Note that the last condition is equivalent to requiring all the entries of the eigenvector $v_{\mathrm{max}}$ to be non-zero and of the same sign.

\subsection{Stability calculations}

In this section, we will prove
\begin{theorem}\label{thm:stab_thm}
Let $\delta^*=\frac{q+1-\jnodes}{\jnodes}$. If $q=\jnodes$, then the corresponding heteroclinic cycle is f.a.s. if $\delta>\delta^*$, and is unstable otherwise. All heteroclinic cycles with $q>\jnodes$ are unstable.
\end{theorem}

Note that $\delta>\delta^*$ is equivalent to $\gamma>\frac{\log r}{\hat{x}}\frac{\nnodes-\jnodes}{\jnodes}$.

We prove theorem~\ref{thm:stab_thm} by presenting results about the eigenvalues of the matrix $M$. First note that the characteristic polynomial of $M$ is
\begin{equation}\label{eq:charpoly}
P(\lambda) = \lambda^q+\lambda^{q-1}+\dots +\lambda^\jnodes -\delta(\lambda^{\jnodes-1}+\dots+\lambda+1)=0
\end{equation}
(this follows, e.g. from Claim 1 in Postlethwaite and Dawes, 2010~\cite{postlethwaite2010resonance}, page 629), and recall that $q=\jnodes(p-1)$, $p \ge 2$, $\jnodes\ge 1$. To establish properties of the roots of the polynomial $P(\lambda)$ we will appeal to three classical results. 
\begin{theorem}[Descartes Rule of Signs]\label{thm:descartes}
For a polynomial with real coefficients, ordered by descending variable exponent, the number of positive roots  of the polynomial is either equal to the number of sign changes between consecutive (nonzero) coefficients, or is less than it by an even number. A root of multiplicity $k$ is counted as $k$ roots.
\end{theorem}
Of the $q$ complex roots of $P(\lambda)$, exactly one, $r_+$, is real and positive, by theorem \ref{thm:descartes}. Next, given a polynomial
$$
f(x) = a_n x^n + a_{n-1}x^{n-1} + \ldots + a_1 x + a_0,
$$
with $a_n \ne 0$, define
\begin{eqnarray*}
f_+(x) &= |a_n| x^n + |a_{n-1}|x^{n-1} + \ldots + |a_1| x - |a_0|& \\
f_-(x) &= |a_n| x^n - |a_{n-1}|x^{n-1} - \ldots - |a_1| x - |a_0|, 
\end{eqnarray*}
and note that theorem \ref{thm:descartes} shows that $f_+(x)$ has exactly one real positive root, $\hat{f_+}>0$, and $f_-(x)$ has exactly one real positive root $\hat{f_-}>0$.

\begin{theorem}[Cauchy \cite{cauchy1829leccons}]\label{thm:cauchy}
All zeros $z$ of $f(x)$ lie in the annular region  
$$
\hat{f_+} \le | z | \le \hat{f_-}.
$$
\end{theorem}

We will also use
\begin{theorem}[Rouch\'e \cite{rouche1866memoire}]\label{thm:rouche}
Let $f$ and $g$ be functions analytic inside and on a simple closed contour $C$, and suppose $|g(z)| < |f(z)|$ on $C$. Then both $f$ and $f+g$ have the same number of zeros inside $C$ (with each zero counted as many times as its multiplicity). 
\end{theorem}

We will prove
\begin{proposition}\label{prop:roots}
The polynomial $P(\lambda)$ (with $\delta>0$) has roots satisfying the following:
\begin{enumerate}
\item When $q=\jnodes$, $r_+>1$ is the root of largest magnitude.
\item When $q>\jnodes$ we have the following cases:
\begin{enumerate}
\item If $\jnodes$ is odd and $q$ is even, the root of largest magnitude is real and negative. Call this root $r_-$; then there are $q-1$ roots inside $|\lambda|=r_-$, one of which is $r_+$, and $q-2$ of which are complex. 
\item If $\jnodes$ and $q$ are both even, or if $q$ is odd, then there are $\jnodes-1$ roots inside $|\lambda|=\max \{r_+,1\}$, of which exactly one, $r_-$, is real if and only if $\jnodes$ is even, and there are $q-\jnodes$ complex roots outside $|\lambda|=r_+$.
\end{enumerate}
\end{enumerate}
\end{proposition}

\begin{proof} 

\begin{enumerate}
\item  First, in the case $q=\jnodes$,  $P(\lambda) = P_-(\lambda)$, and so by theorem \ref{thm:cauchy},  all roots of $P(\lambda)$ are bounded in magnitude by $r_+$. Moreover, $r_+>1$ when $\delta > \delta^*$. 

\item When $q>\jnodes$, theorem \ref{thm:cauchy} is no longer of use. It is convenient to study the related polynomial
$$
Q(\lambda) = (\lambda-1)P(\lambda) = \lambda^{q+1} - (1+\delta)\lambda^\jnodes + \delta,
$$
which has the same roots as $P(\lambda)$, plus a root at $\lambda=1$.  Considering $Q'(\lambda)=(q+1)\lambda^q-\jnodes(1+\delta)\lambda^{\jnodes-1}$, and recalling that a double root of a polynomial is also a root of the polynomial's derivative, we see that $r_+=1$ when $\delta = \delta^* = (q+1-\jnodes)/\jnodes$, and that 
\begin{equation}\label{eq:rplus}
Q'(r_+) = (q+1)r_+^q - \jnodes(1+\delta)r_+^{\jnodes-1} >0 
\end{equation}
when $\delta>\delta^*$ (and $Q'(r_+)<0$ when $\delta<\delta^*$). 
\begin{enumerate}
\item In the case $\jnodes$ odd and $q$ even we have
\begin{eqnarray*}
Q(-\lambda) &=& -\lambda^{q+1} + (1+\delta)\lambda^\jnodes + \delta\\
&=& -Q(\lambda) + 2\delta,
\end{eqnarray*}
so $Q(-r_+) = -Q(r_+) + 2\delta = 2\delta >0$. Since $Q(\lambda) \to -\infty$ as $\lambda \to -\infty$, there must exist a real root of $Q(\lambda)$ between $-\infty$ and $-r_+$. Thus $Q(\lambda)$, and hence also $P(\lambda)$, has a real negative root, $r_-$, of greater magnitude than $r_+$. The same argument applies to $Q(-1)$ so $r_-$ is also of greater magnitude than the root at $\lambda=1$. To show that (in this case) $r_-$ is the root of largest magnitude, we apply theorem \ref{thm:rouche}, setting $f(\lambda) = \lambda^{q+1}$ and $g(\lambda) = -(1+\delta)\lambda^\jnodes + \delta$. We have, on the circle of radius $|\lambda| = (1+\epsilon)|r_-|$, 
\begin{eqnarray*}
|f(\lambda)| &=& |(r_-(1+\epsilon))^{q+1}| \\
& > & | (1+\delta)(r_-(1+\epsilon))^\jnodes | + \delta \\
&\ge & | (1+\delta)(r_-(1+\epsilon))^\jnodes  - \delta | \\
& = & | g(\lambda) |, 
\end{eqnarray*}
where the first inequality follows since $|r_-^{q+1}|=|(1+\delta)r_-^\jnodes|+\delta$, and the second by the triangle inequality. Then by theorem \ref{thm:rouche}, $Q(\lambda) = f(\lambda)+g(\lambda)$ has the same number of complex zeros inside that circle as $f(\lambda)$, namely $q+1$. Hence $r_-$ is the root of largest modulus.

\item In the other cases we will again apply theorem \ref{thm:rouche}. First consider $\delta<\delta^*$, so $r_+<1$ and take $f(\lambda) = -(1+\delta)\lambda^\jnodes$ and $g(\lambda) = \lambda^{q+1}+\delta$. For $\epsilon < 2(1+q-\jnodes(1+\delta))/(q(q+1))$ (observing that $\epsilon>0$ since $\delta<\delta^*$) we have, on the circle $|\lambda| = 1-\epsilon$,
\begin{eqnarray*}
|f(\lambda)| &=& |-(1+\delta)(1-\epsilon)^\jnodes| \\
&>&(1+\delta)(1-\jnodes\epsilon) \\
&=& 1+\delta + \epsilon(-\jnodes(1+\delta)) \\
&>& 1+\delta + \epsilon( -(q+1)+\epsilon q(q+1)/2) \\
&=& 1-\epsilon (q+1) + \epsilon^2 q(q+1)/2 + \delta \\
&>& (1-\epsilon)^{q+1} + \delta \\
&>& | (1-\epsilon)^{q+1} + \delta | \\
&=& | g(\lambda) |
\end{eqnarray*}

Similarly, if $\delta>\delta^*$, so $r_+>1$, we inspect $f(\lambda)$ and $g(\lambda)$ on the circle $|\lambda| = (1-\epsilon)r_+$, with $\epsilon < 2(1+q-\jnodes(1+\delta)/r^{q+1-\jnodes})/q(q+1)$ (which is again positive by (\ref{eq:rplus})). We have, similarly,
\begin{eqnarray*}
|f(\lambda)| &=& |-(1+\delta)(1-\epsilon)^\jnodes r_+^\jnodes| \\
&>&(1+\delta)(1-\jnodes\epsilon) r_+^\jnodes\\
&=& r_+^{q+1}+\delta + \epsilon r_+^\jnodes (-\jnodes(1+\delta)) \\
&>& (1-\epsilon)^{q+1}r_+^{q+1} + \delta \\
&>& | (1-\epsilon)^{q+1}r^{q+1} + \delta | \\
&=& | g(z) |.
\end{eqnarray*}
\end{enumerate}
Theorem \ref{thm:rouche} implies that $Q(\lambda)$ has $\jnodes$ roots inside the circle $|\lambda | = 1-\epsilon$ (resp. $|\lambda | = (1-\epsilon)r_+)$ if $\delta<\delta^*$ (resp. $\delta > \delta^*)$, and so $P(\lambda)$ has $\jnodes-1$ roots inside those circles. 
\end{enumerate}

\end{proof}


\begin{proof}[Proof of theorem \ref{thm:stab_thm}]
When $q=\jnodes$, proposition~\ref{prop:roots} shows that $\lambda_{\max}$ for the matrix $M$ is real and is $>1$ when $\delta$ is sufficiently large (when $\delta>1/\jnodes$). Moreover an easy calculation shows that the corresponding eigenvector $v_{\max}$ is given by
\begin{equation}
v_{\max}=\begin{pmatrix}
1\\
\frac{\delta}{r} \left( \sum_{k=0}^{\jnodes-2} \frac{1}{r^k}  \right) \\
\frac{\delta}{r} \left( \sum_{k=0}^{\jnodes-3} \frac{1}{r^k}  \right)\\
\vdots \\
\frac{\delta}{r} \left( 1+ \frac{1}{r} \right) \\
\frac{\delta}{r}
\end{pmatrix},
\end{equation}
which clearly has all entries non-zero and the same sign, and lemma \ref{lemma:fas} confirms that this state is fragmentarily asymptotically stable. 

When $q>\jnodes$ the conditions of lemma \ref{lemma:fas} are not met, as $\lambda_{\max}$ is no longer real and positive.

\end{proof}



\subsection{Appearance of instabilities}

As described in the section above, many of the heteroclinic cycles we find are unstable. However, if initial conditions are carefully chosen, then the cycles can be observed for reasonably long times in numerical simulations. Specifically, suppose that a heteroclinic cycle $H$ has transition matrix $M$, with eigenvalues and corresponding eigenvalues $\lambda_i$ and $v_i$. Suppose that the heteroclinic cycle is unstable, so that the eigenvalue with largest magnitude, which without loss of generality we assume to be $\lambda_1$, does not satisfy the conditions of lemma~\ref{lemma:fas}. Assume further that $\lambda_2$ \emph{does} satisfy the conditions of lemma~\ref{lemma:fas}. Then, if we choose initial conditions $X_0=c_2v_2$, then the forward trajectory will remain close to $H$. 

In numerical simulations, of course, errors accumulate, and the trajectory can only remain close to $H$ for a finite time. In figure~\ref{fig:pred}(a), we show an example of a trajectory which remains close to an unstable cycle for a long time. Here, the cycle in question is the one between fixed points with one node active in the $(5,1)$-graph. In figure~\ref{fig:pred}(b) we show the coordinates from (a) at the bottom of each `valley' in the time series: this corresponds to the coordinates at the transition between epochs and hence the coordinates $X_j$. For the particular transition matrix for this cycle with the noted parameters, $\lambda_1$ and $\lambda_2$ satisfy the assumptions given above, and $\lambda_1$ is complex. Using these values for $\lambda_1$ and $\lambda_2$, we use least squares to estimate the values of $c_1$, $c_2$ and $c_3$ to fit the curve $X_j=c_1\lambda_2^j+c_2|\lambda_1|^j \cos(j\arg(\lambda_1)+c_3)$ to the obtained data. The dashed line in (b) shows the curve $X_j=c_1\lambda_2^j$, that is, the data that would be expected if there were no numerical error and we were able to start exactly on the required eigenvector. The solid curve includes the second term and is clearly an excellent fit to the data points.

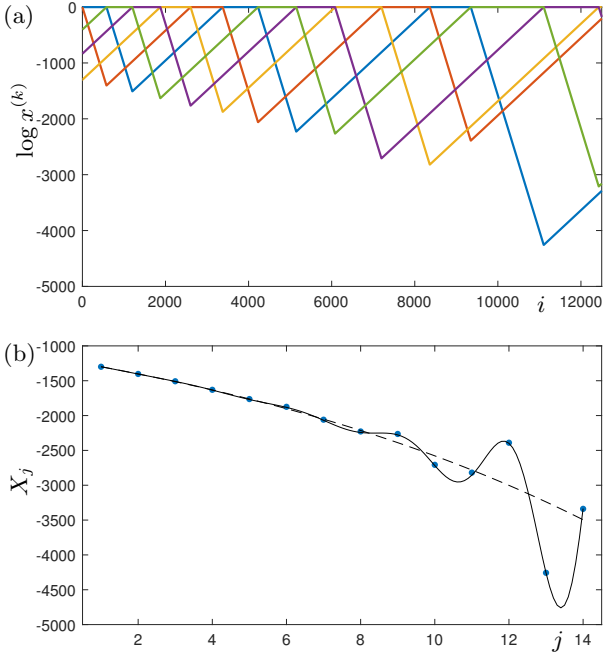
\begin{figure}
\setlength{\unitlength}{1mm}
\begin{center}
\begin{picture}(85,90)(0,0)
\put(0,0){\includegraphics[trim=5mm 5mm 5mm 5mm,clip=true,width=85mm]{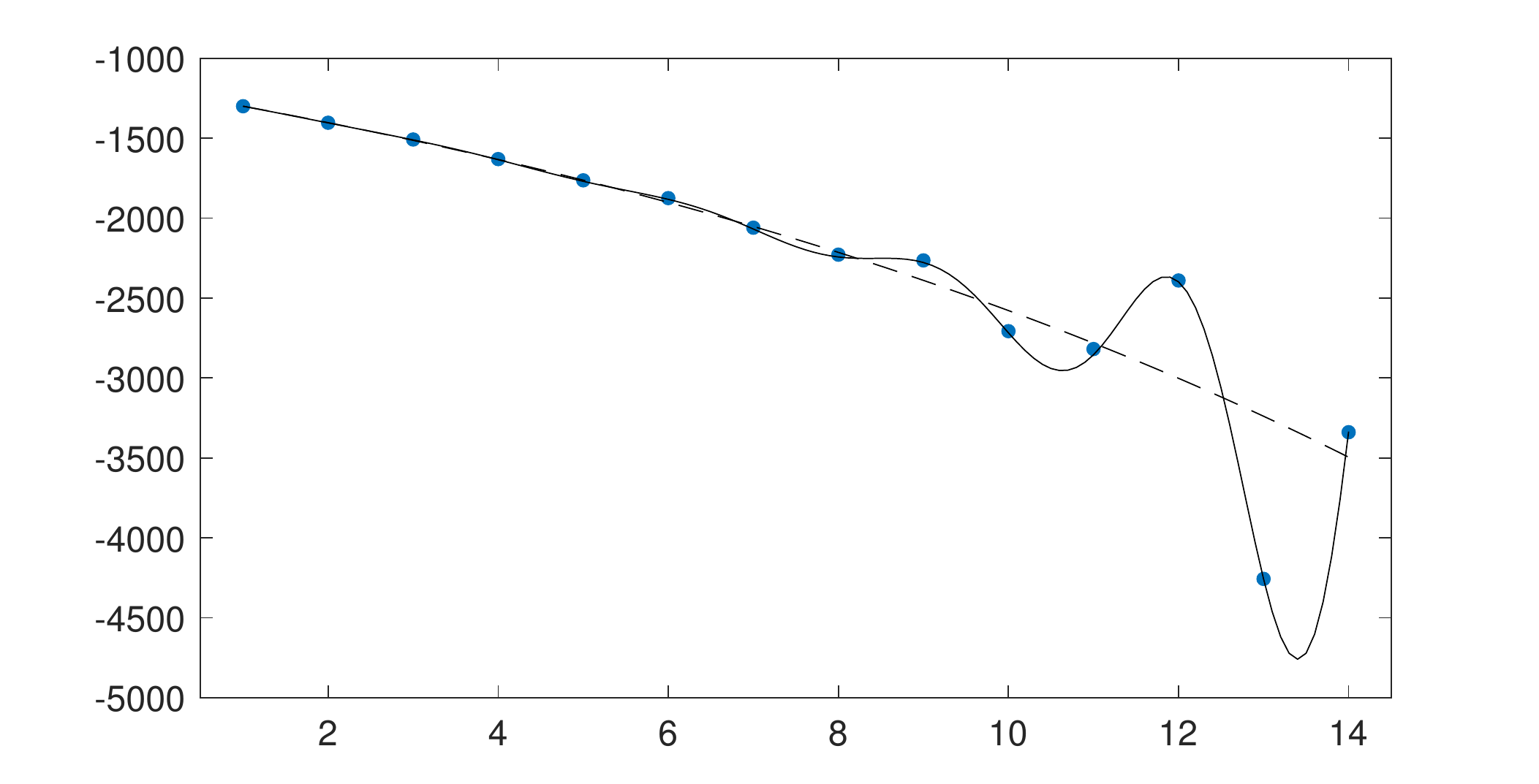}}
\put(0,45){\includegraphics[trim=5mm 5mm 5mm 5mm,clip=true,width=85mm]{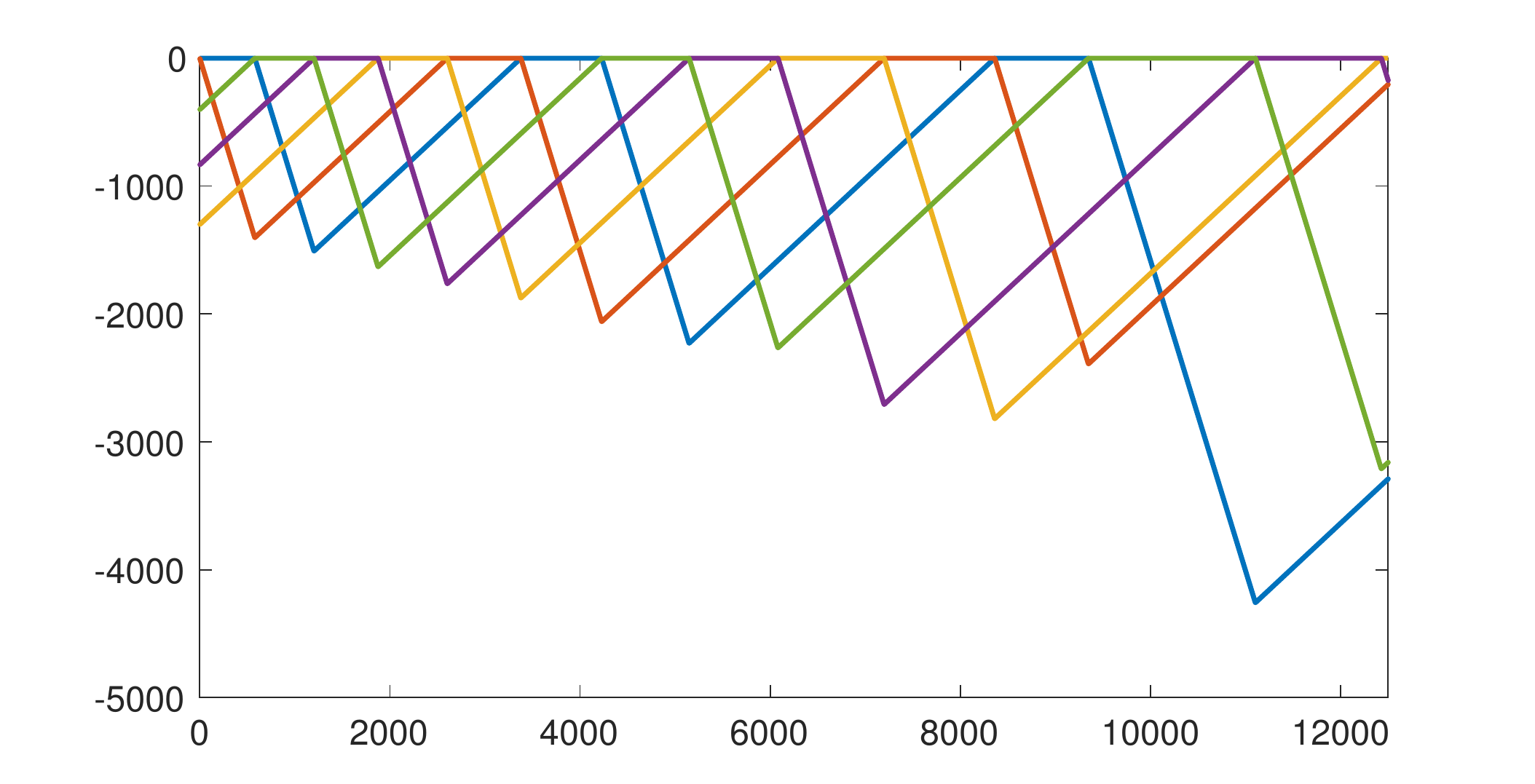}}

\put(0,65){\rotatebox{90}{$\log x^{(k)}$}}
\put(0,20){\rotatebox{90}{$X_j$}}

\put(72,0){$j$}
\put(70,44.5){$i$}

\put(-1,83){(a)}
\put(-1,38){(b)}

\end{picture}
\end{center}
\caption{\label{fig:pred} (a) The figure shows a trajectory of equation~\eqref{eq:five} with initial condition near the unstable heteroclinic cycle between fixed points with one node active. The components $x^{(1)},\dots,x^{(5)}$ are represented by the colours blue, red, yellow, purple and green, respectively. Parameters are $r=2$, $\gamma=6.24$. The blue dots in panel (b) show the coordinates at the bottom of each of the `valleys' in the time series in panel (a). The dashed curve is $X_j=c_1\lambda_2^j$. The solid curve is  $X_j=c_1\lambda_2^j+c_2|\lambda_1|^j \cos(j\arg(\lambda_1)+c_3)$. See text for more details. 
}
\end{figure}

\section{Analysis of ring graph with $m$-nearest neighbour coupling}
\label{sec:nm}

In this section we expand on our results from the previous section to discuss ring graphs with $m$-nearest neighbour coupling ($m<\nnodes/2$). We find that, depending on the number of nodes, $\nnodes$, in the graph, and the number $m$ of neigbours coupled, different types of heteroclinic networks can arise in the dynamics. Some of these have dynamics which can be described using the same methods as in the previous section, and some of these are more complex. We refer to the $\nnodes$-node graph, with $m$-nearest neighbour coupling as the \emph{$(\nnodes,m)$-graph}.

\begin{figure}
\begin{center}
\setlength{\unitlength}{1mm}
\begin{picture}(50,100)(-3,0)
\put(0,60){\includegraphics[width=5cm]{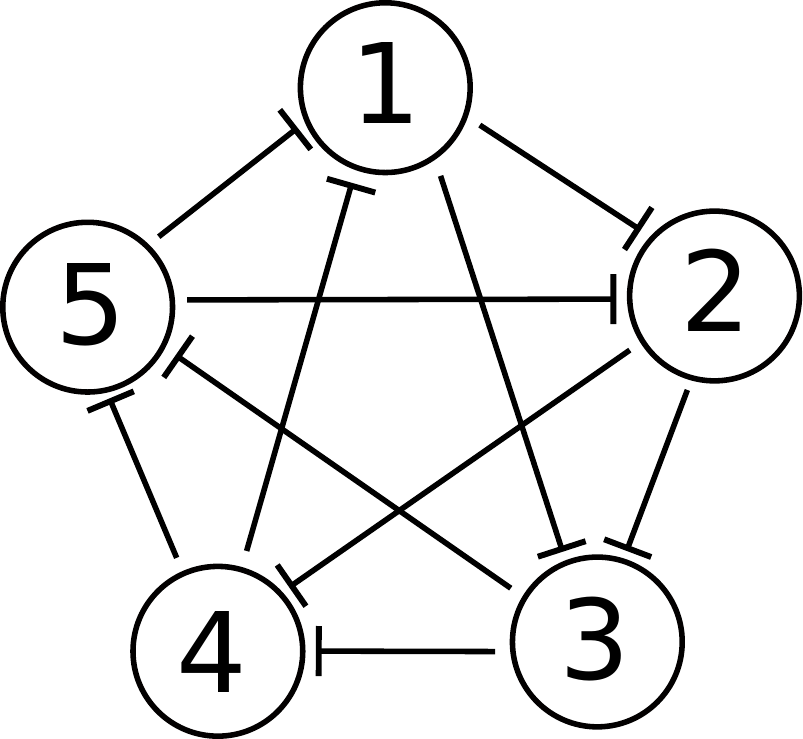}}
\put(0,0){\includegraphics[width=5cm]{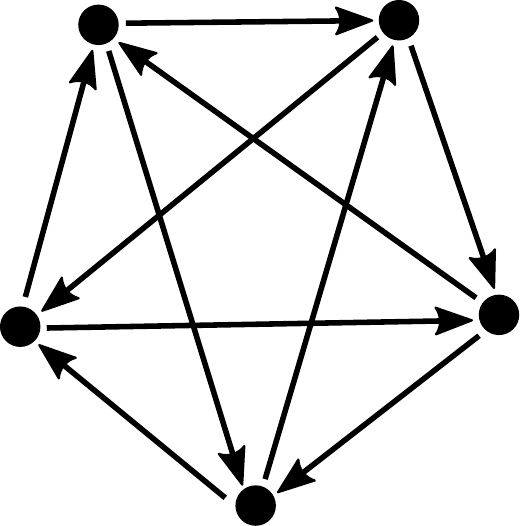}}

\put(41,47){$\xi_1$}
\put(4,47){$\xi_3$}
\put(47,15){$\xi_4$}
\put(1,14.5){$\xi_5$}
\put(17,1){$\xi_2$}

\put(-3,100){(a)}
\put(-3,47){(b)}

\end{picture}
\end{center}
\caption{\label{fig:five2} The figure shows, in panel (a), the physical network of nodes for the $(5,2)$-graph, and in panel (b), the resulting heteroclinic network between fixed points.
}
\end{figure}
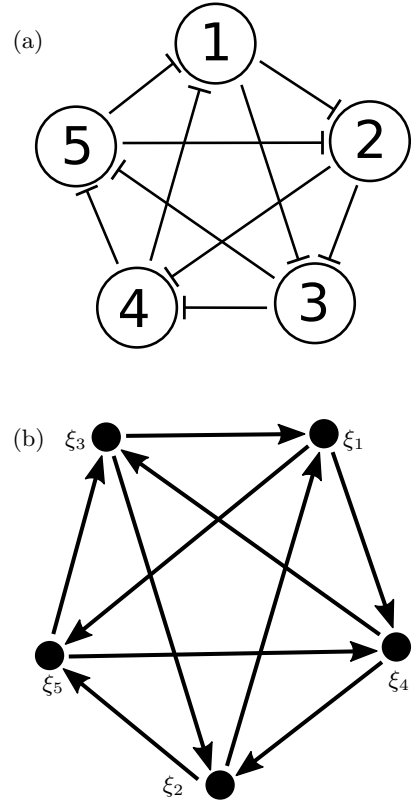
The smallest graph which falls into this category (with $m\neq 1$)  is the a five-node graph, with $m=2$, shown in figure~\ref{fig:five2}(a). In this example, it is not possible to have any fixed points which have more than one component non-zero, and the network of heteroclinic connections between these fixed points is shown on the right-hand side of figure~\ref{fig:five2}. There are two-subcycles in this network between five fixed points, and the transition matrices for these cycles can be found using the methods in the previous section. We find, for the cycle $\xi_1\rightarrow \xi_5 \rightarrow \xi_4 \rightarrow \xi_3 \rightarrow \xi_2 $ that the transition matrix is 
\[
\begin{pmatrix}
-1 & 1 & 0 \\
\delta & 0 & 1 \\
\delta & 0 & 0
\end{pmatrix}.
\]
The eigenvalues can be found explicitly as $\pm\sqrt{\delta}$, $-1$, and thus this cycle can never be fragmentarily asymptotically stable.
For the cycle $\xi_1\rightarrow \xi_4 \rightarrow \xi_2 \rightarrow \xi_5 \rightarrow \xi_3 $ the transition matrix is
\[
\begin{pmatrix}
\delta & 1 & 0 \\
-1 & 0 & 1 \\
\delta & 0 & 0
\end{pmatrix}.
\]
which has eigenvalues $\delta$, $\pm i$, but the eigenvector for the eigenvalue $\delta$ has a zero in the second component, and so this cycle can also never be fragmentarily asymptotically stable. There are other routes trajectories can take whilst still approaching the network: in fact, this network is equivalent to the Rock-Paper-Scissors-Lizard-Spock network investigated by Postlethwaite and Rucklidge (for ODEs)~\cite{postlethwaite2020}, which has some very complicated dynamics: see figure~\ref{fig:rpsls_ts} for a typical time series.

\begin{figure}
\setlength{\unitlength}{1mm}
\begin{center}
\begin{picture}(85,40)(0,0)
\put(0,0){\includegraphics[trim=5mm 5mm 5mm 5mm,clip=true,width=85mm]{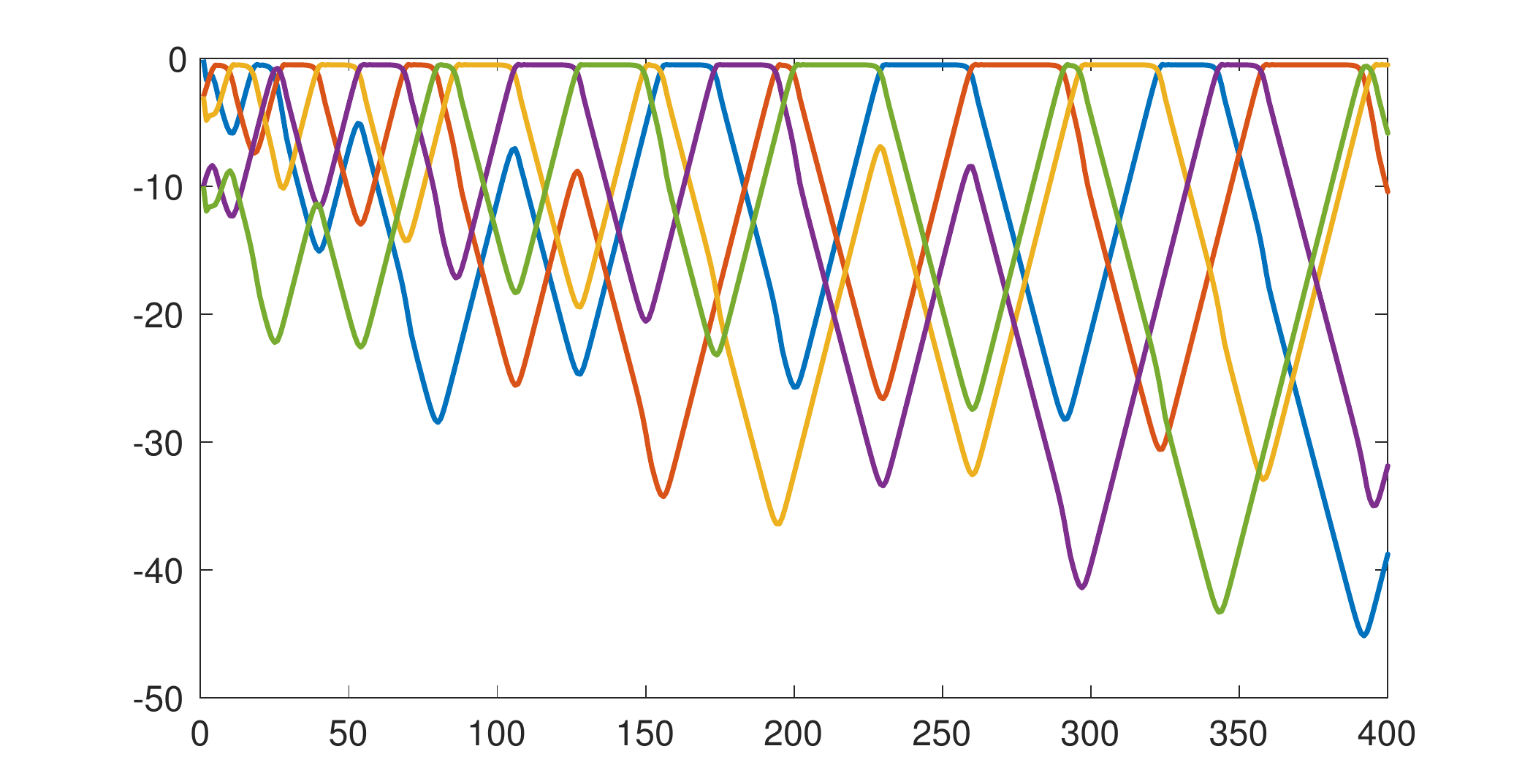}}

\put(0,20){\rotatebox{90}{$\log x^{(i)}$}}
\put(74,0){$i$}

\end{picture}
\end{center}
\caption{\label{fig:rpsls_ts} The figure show a typical trajectory for the network shown in figure~\ref{fig:five2}. The components $x^{(1)},\dots,x^{(5)}$ are represented by the colours blue, red, yellow, purple and green, respectively. Parameters are $r=2.5$, $\gamma=3$. Although the network appears to be attracting (trajectories get closer to the fixed points as time increases), the sequence in which the fixed points are visited is irregular and complicated.}
\end{figure}

As a second example, consider the seven-node graph with two-nearest neighbour coupling, shown in figure~\ref{fig:seven2}(a). In this example, there are seven fixed points which have exactly one non-zero component, and seven with exactly two non-zero components. The heteroclinic network between these fixed points is shown in figure~\ref{fig:seven2}(b). Note the similarity in structure between this network and the network shown in figure~\ref{fig:five}(b). The stability of the cycles between the fixed points with either one or two non-zero components can be computed in exactly the same way as shown previously, and we find that the cycle between the fixed points with two non-zero components can be stable if $\delta$ is large enough, but the other cycle cannot.

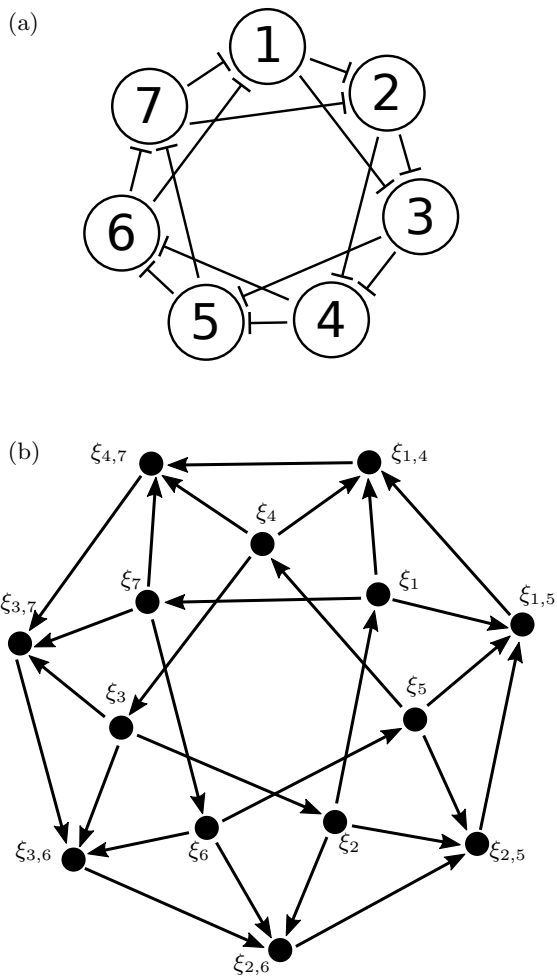
\begin{figure}
\begin{center}
\setlength{\unitlength}{1mm}
\begin{picture}(75,130)(-3,0)
\put(10,80){\includegraphics[width=5cm]{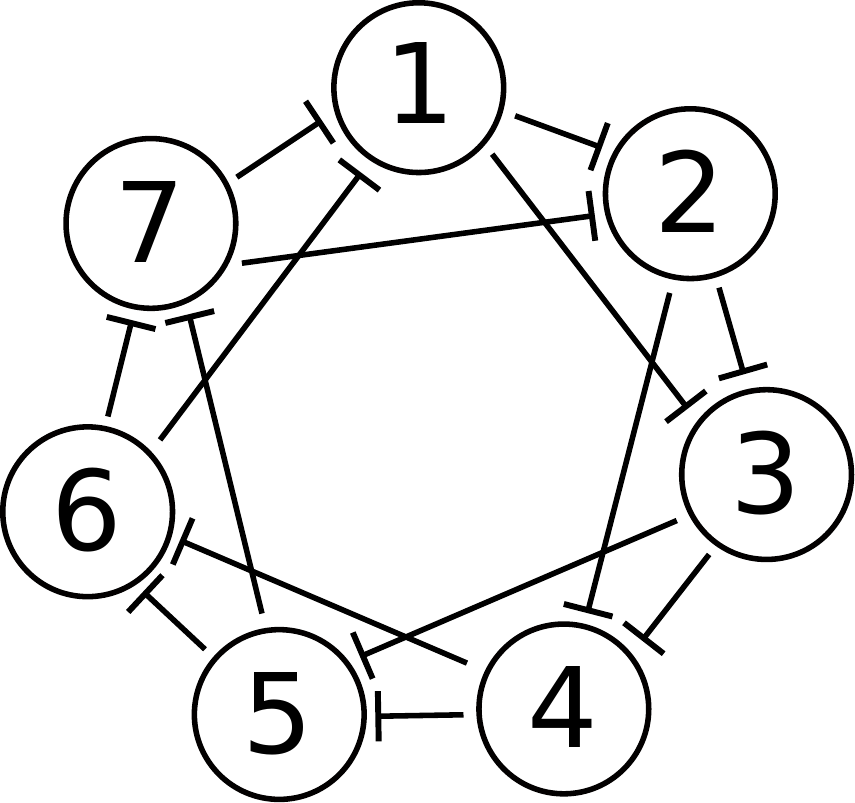}}
\put(0,0){\includegraphics[width=7cm]{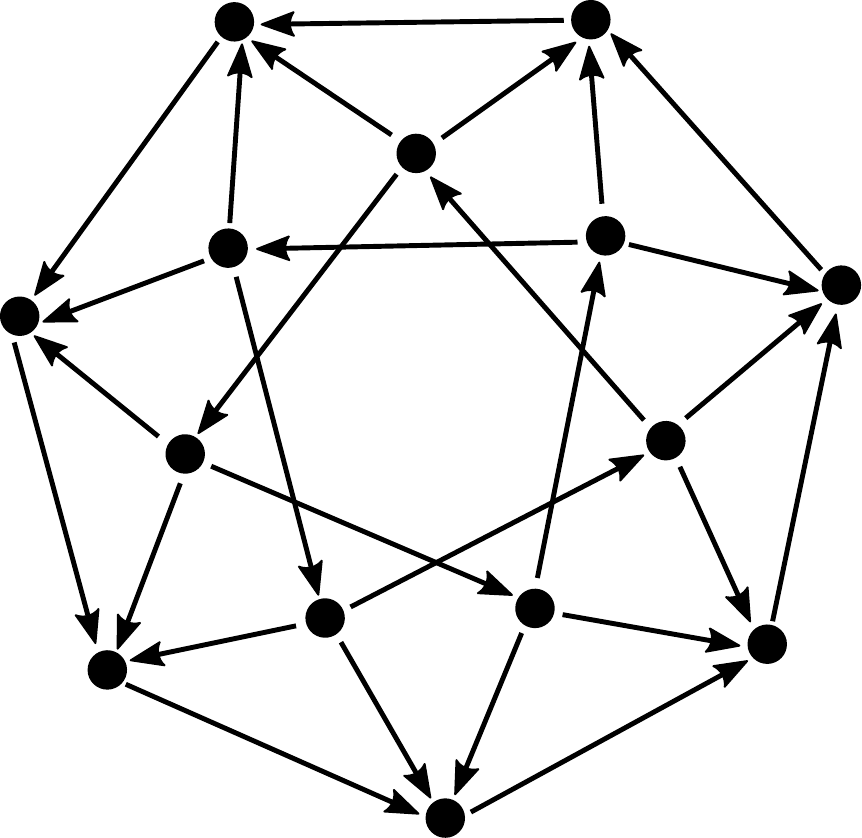}}

\put(51,67){$\xi_{1,4}$}
\put(11,67){$\xi_{4,7}$}
\put(-1,47){$\xi_{3,7}$}
\put(68,48){$\xi_{1,5}$}
\put(1,14){$\xi_{3,6}$}
\put(64,14){$\xi_{2,5}$}
\put(30,-1){$\xi_{2,6}$}

\put(52,50){$\xi_{1}$}
\put(15,50){$\xi_{7}$}
\put(33,59){$\xi_{4}$}
\put(53,36){$\xi_{5}$}
\put(13,35){$\xi_{3}$}
\put(44,15){$\xi_{2}$}
\put(24,14){$\xi_{6}$}

\put(0,124){(a)}
\put(0,67){(b)}

\end{picture}
\end{center}
\caption{\label{fig:seven2} The figure shows, in panel (a), the physical network of nodes for the $(7,2)$-graph, and in panel (b), the resulting heteroclinic network between fixed points.
}
\end{figure}

As in the examples of the $(\nnodes,1)$ graphs, when transitioning from one fixed point to another along a heteroclinic connection, the number of active nodes may increase, but it can never decrease. We refer to those fixed points with the largest number of active nodes as the \emph{maximally-active} fixed points, and in the following, discuss the possible network of heteroclinic connections between these nodes for a general $(\nnodes,m)$-graph. We have the following:
\begin{itemize}
\item If $\nnodes = 0 \mod (m+1)$, then all maximally-active fixed points are asymptotically stable, and there are no heteroclinic connections.
\item If $\nnodes=1 \mod (m+1)$, then all maximally-active fixed points have an unstable manifold of dimension one, and there exists a heteroclinic cycle between the fixed points. The transition matrix for this heteroclinic cycle takes the form of $M$ in equation~\eqref{eq:M}, with all rows starting with a $\delta$.  It can thus be asymptotically stable if $\delta$ is sufficiently large.
\item If $\nnodes=p \mod (m+1)$, $p\neq 0,1$, then all maximally-active fixed points have an unstable manifold of dimension $p$, and there exists a heteroclinic network between the fixed points. We conjecture that this network can be asymptotically stable for large enough $\delta$, but may have complex dynamics.
\end{itemize}


\section{Discussion}
\label{sec:disc}

In this paper we have shown that heteroclinic networks can typically arise in the phase space dynamics of certain types of symmetric (physical space) graphs with inhibitory coupling. We further showed that at most one of the subcycles can be stable for an open set of parameters, and hence observable in simulations. Many studies of coupled map lattices and complex networks seek asymptotic behaviour described by a Sinai-Ruelle-Bowen (SRB) invariant measure\cite{tanzi2021existence}. However the dynamics associated with a stable heteroclinic cycle preclude this behaviour --- the dynamics is not ergodic, and long-term averages do not converge. In particular, averaged observed quantities such as Lyapunov exponents are ill-defined, and will oscillate at a progressively slower rate.

From this work arises the more general question of whether or not a stable heteroclinic cycle is likely to be found in the corresponding phase space network of a randomly generated physical network of nodes. We performed some preliminary investigations on this question numerically, for randomly generated Erdos--R\'enyi graphs (where each edge exists with some fixed probability). We find that the probability of existence of heteroclinic cycles in the phase space network increases both as the number of nodes in the physical network increases, and also as the density of edges in the physical network decreases. However, even in cases where the probability of existence of heteroclinic cycles is very high, there is also a very high chance of the existence of a stable fixed point in the phase space. Thus, the question of the stability of the heteroclinic cycle is important in determining whether or not the heteroclinic cycle, and associated slowing down of trajectories, will be observed in the phase space associated with a randomly generated graph. 

There methods we describe in this paper can be used to determine the stability of any specific heteroclinic cycle, but as yet it is not clear how one would determine the likelihood of a heteoclinic cycle to be stable in such a randomly-generated network.

In this work, we consider specific dyanmics for each single node in the directed graph in physical space; specifically, we suppose that there is only an `on' state, and an `off' state. If more general dyanmics are allowed, then other types of heteroclinic cycles can be found in ring graphs~\cite{buono2000heteroclinic}. 

An obvious extension of this work would include different types of coupling and/or different dynamics in the uncoupled nodes. For instance, a situation which might better exemplify neuronal dyanmics could include both inhibitory and excitatory types of connections, and nodes could require a `kick' from an `on' excitatory connection in order to leave a stable zero state. Networks of this type were investigated by Ashwin and Postlethwaite~\cite{ashwin2021excitable}, although they made no attempt to classify the possible heteroclinic networks which could occur. Further work on this avenue of investigation is ongoing.


%
%

%

\begin{acknowledgments}
CMP acknowledges support from the Marsden Fund Council from New Zealand Government funding, managed by The Royal Society Te Ap\={a}rangi.
\end{acknowledgments}

\bibliography{hetpaper}

\end{document}